\documentclass[12pt]{amsart}
\usepackage{mathrsfs}
\usepackage{amsfonts}
\usepackage{enumitem}
\usepackage[all]{xy}
\usepackage{amssymb,amsfonts}
\begin{document}
\theoremstyle{plain}
\newtheorem{thm}{Theorem}[section]
\newtheorem{theorem}[thm]{Theorem}
\newtheorem{lemma}[thm]{Lemma}
\newtheorem{remark}[thm]{Remark}
\newtheorem{corollary}[thm]{Corollary}
\newtheorem{proposition}[thm]{Proposition}
\newtheorem{conjecture}[thm]{Conjecture}
\newtheorem{definition}[thm]{Definition}
\newtheorem{acknowledgement}[thm]{Acknowledgement}
\theoremstyle{definition}
\newtheorem{construction}[thm]{Construction}
\newtheorem{notations}[thm]{Notations}
\newtheorem{question}[thm]{Question}
\newtheorem{problem}[thm]{Problem}
\newtheorem{remarks}[thm]{Remarks}
\newtheorem{claim}[thm]{Claim}
\newtheorem{assumption}[thm]{Assumption}
\newtheorem{assumptions}[thm]{Assumptions}
\newtheorem{properties}[thm]{Properties}
\newtheorem{example}[thm]{Example}
\newtheorem{comments}[thm]{Comments}
\newtheorem{blank}[thm]{}
\newtheorem{observation}[thm]{Observation}
\newtheorem{defn-thm}[thm]{Definition-Theorem}
\renewcommand{\bar}{\overline}
\newcommand{\eps}{\varepsilon}
\newcommand{\pa}{\partial}
\renewcommand{\phi}{\varphi}
\newcommand{\wt}{\widetilde}
\newcommand{\oo}{\mathcal O}
\newcommand{\ka}{K\"ahler }
\newcommand{\kar}{K\"ahler-Ricci}
\newcommand{\C}{{\mathbb C}}
\newcommand{\Z}{{\mathcal{Z}_{m}}}
\newcommand{\R}{{\mathbb R}}
\newcommand{\N}{{\mathbb N}}
\newcommand{\Q}{{\mathbb Q}}
\newcommand{\pz}{\partial_z}
\newcommand{\pzb}{\partial_{\bar z}}
\newcommand{\M}{{\mathcal M}}
\newcommand{\T}{{\mathcal T}}
\newcommand{\NN}{{\mathcal N}}
\newcommand{\XX}{{\widetilde{\mathfrak X}}}
\newcommand{\pzs}{\partial_{z_s}}
\newcommand{\pzbs}{\partial_{\bar{z_s}}}
\newcommand{\pp}{{\mathbb P}}
\newcommand{\ke}{K\"ahler-Einstein }
\newcommand{\hh}{{\mathcal H}}
\newcommand{\kk}{{\mathcal K}}
\newcommand{\tei}{Teichm\"uller }
\newcommand{\h}{{\mathbb H}}
\renewcommand{\tilde}{\widetilde}
\newcommand{\p}{{\Phi}}
\newcommand{\g}{{\mathfrak g}}
\newcommand{\kkk}{{\mathfrak k}}
\newcommand{\mkp}{{\mathfrak p}}
\newcommand{\lb}{\left (}
\newcommand{\rb}{\right )}
\newcommand\independent{\protect\mathpalette{\protect\independenT}{\perp}}
\def\independenT#1#2{\mathrel{\rlap{$#1#2$}\mkern2mu{#1#2}}}

\def\dW{\mbox{diff\:}\times \mbox{Weyl\:}}
\def\End{\operatorname{End}}
\def\Hom{\operatorname{Hom}}
\def\Aut{\operatorname{Aut}}
\def\Diff{\operatorname{Diff}}
\def\im{\operatorname{im}}
\def\tr{\operatorname{tr}}
\def\Pr{\operatorname{Pr}}
\def\O{\mathcal{O}}
\def\CP{\mathbb{C}\mathbb{P}}
\def\P{\bf P}
\def\Q{\bf Q}
\def\R{\bf R}
\def\C{\mathbb{C}}
\def\H{\bf H}
\def\Hil{\mathcal{H}}
\def\proj{\operatorname{proj}}
\def\id{\mbox{id\:}}
\def\a{\alpha}
\def\b{\beta}
\def\c{\gamma}
\def\p{\partial}
\def\f{\frac}
\def\i{\sqrt{-1}}
\def\t{\tau}
\def\T{\mathcal{T}}
\def\Kahler{K\"{a}hler\:}
\def\w{\omega}
\def\X{\mathfrak{X}}
\def\K{\mathcal {K}}
\def\m{\mu}
\def\M{\mathcal {M}}
\def\v{\nu}
\def\D{\mathcal{D}}
\def\U{\mathcal {U}}
\def\Omegak{\frac{1}{k!}\bigwedge\limits^k\mu\lrcorner\Omega}
\def\Omegakp{\frac{1}{(k+1)!}\bigwedge\limits^{k+1}\mu\lrcorner\Omega}
\def\Omegakpp{\frac{1}{(k+2)!}\bigwedge\limits^{k+2}\mu\lrcorner\Omega}
\def\Omegakm{\frac{1}{(k-1)!}\bigwedge\limits^{k-1}\mu\lrcorner\Omega}
\def\Omegakmm{\frac{1}{(k-2)!}\bigwedge\limits^{k-2}\mu\lrcorner\Omega}
\def\Omegakk{\Omega_{i_1,i_2,\cdots,i_k}}
\def\Omegakkp{\Omega_{i_1,i_2,\cdots,i_{k+1}}}
\def\Omegakkpp{\Omega_{i_1,i_2,\cdots,i_{k+2}}}
\def\Omegakkm{\Omega_{i_1,i_2,\cdots,i_{k-1}}}
\def\Omegakkmm{\Omega_{i_1,i_2,\cdots,i_{k-2}}}
\def\mukm{\frac{1}{(k-1)!}\bigwedge\limits^{k-1}\mu}
\def\sumk{\sum\limits_{i_1<i_2<,\cdots,<i_k}}
\def\sumkm{\sum\limits_{i_1<i_2<,\cdots,<i_{k-1}}}
\def\sumkmm{\sum\limits_{i_1<i_2<,\cdots,<i_{k-2}}}
\def\sumkp{\sum\limits_{i_1<i_2<,\cdots,<i_{k+1}}}
\def\sumkpp{\sum\limits_{i_1<i_2<,\cdots,<i_{k+2}}}
\def\Omegakb{\Omega_{i_1,\cdots,\bar{i}_t,\cdots,i_k}}
\def\Omegakmb{\Omega_{i_1,\cdots,\bar{i}_t,\cdots,i_{k-1}}}
\def\Omegakpb{\Omega_{i_1,\cdots,\bar{i}_t,\cdots,i_{k+1}}}
\def\Omegakt{\Omega_{i_1,\cdots,\tilde{i}_t,\cdots,i_k}}


\title[]{A GLOBAL TORELLI THEOREM FOR CALABI-YAU MANIFOLDS}
      
        \author{Kefeng Liu}
        \dedicatory{Center of Math Sciences, Zhejiang University;
           Department of Mathematics, UCLA}
        \address{Center of Mathematical Sciences, Zhejiang University, Hangzhou, Zhejiang 310027, China;
                 Department of Mathematics,University of California at Los Angeles,
                Los Angeles, CA 90095-1555, USA}
        \email{liu@math.ucla.edu, liu@cms.zju.edu.cn}
          \author{Yang Shen}
        \address{ Center of Mathematical Sciences, Zhejiang University, Hangzhou, Zhejiang 310027, China  }
        \email{syliuguang2007@163.com}     
           \author{Andrey Todorov}
        \dedicatory{}
        \address{Department of Mathematics, University of California, Santa Cruz, CA 95064, USA}
        \email{todorov.andrey@gmail.com}

        \begin{abstract}
We describe the proof that the period map from the Torelli space of Calabi-Yau manifolds to the classifying space
of polarized Hodge structures is an embedding. The proof is based on the
constructions of holomorphic affine structure on the Teichm\"uller space and  Hodge metric completion of the Torelli space. A canonical global holomorphic section of the holomorphic $(n, 0)$ class on the Teichm\"uller space is constructed.

        \end{abstract}
    \maketitle
\tableofcontents

\section{Introduction}\label{Section introduction}
In this paper we study the global properties of the period map from
the Torelli space of Calabi-Yau
manifolds to the classifying space of polarized Hodge structures.
Our method is based on the existence of a holomorphic affine structure on the Teichm\"uller space, and the construction of a completion space of the Torelli space  with a compatible affine structure on it.

Although our method works for more general cases, for simplicity we
will restrict our discussion to Calabi-Yau manifolds. More precisely
a compact projective manifold $M$ of complex dimension $n$ with $n\geq 3$ is called
Calabi-Yau in this paper, if it has a trivial canonical bundle and
satisfies $H^i(M,\mathcal{O}_M)=0$ for $0<i<n$. 
We fix a lattice $\Lambda$ with an pairing $Q_{0}$, where $\Lambda$ is isomorphic to $H^n(M_{0},\mathbb{Z})/\text{Tor}$ for some fixed Calabi--Yau manifold $M_{0}$, and $Q_0$ is the intersection pairing.
A polarized and
marked Calabi-Yau manifold is a triple consisting of a Calabi-Yau
manifold $M$, an ample line bundle $L$ over $M$ and a marking $\gamma$ defined as an isometry of the lattices
$$\gamma :\, (\Lambda, Q_{0})\to (H^n(M,\mathbb{Z})/\text{Tor},Q).$$

Let $\mathcal{Z}_m$ be a smooth component containing $M$ of the moduli space of polarized Calabi--Yau manifolds with level $m$ structure with $m\geq 3$, which is constructed by Popp, Viehweg, and Szendr\"oi, for example in Section 2 of \cite{sz}. 
We define the Teichm\"uller space of Calabi--Yau manifolds to be the universal cover of $\mathcal{Z}_m$, which will be proved to be independent of the choice of $m$. 
The Teichm\"uller space will be denoted by $\T$.

The Torelli space $\mathcal{T}'$ is defined to be a connected component of the moduli space of 
polarized and marked Calabi-Yau manifolds containing $M$. 
The Torelli space is also a smooth cover of $\mathcal{Z}_m$. 
Therefore both $\mathcal{T}$ and $\T'$
are connected smooth complex manifolds of complex dimension
\begin{align*}
\dim_{{\mathbb{C}}}{\mathcal{T}'}=h^{n-1,1}(M)=N,
\end{align*}
where $h^{n-1,1}(M)=\dim_{\mathbb{C}} H^{n-1,1}(M)$ with
$H^{n-1,1}(M)$ the $(n-1,1)$-Dolbeault cohomology group of $M$. See
Section \ref{section construction of Tei} for details.

There is a universal family $\mathcal{X}_{\Z}\to \Z$ over $\Z$ constructed in \cite{sz}. Then we can pull back the universal family $\mathcal{X}_{\Z}\to \Z$ via the covering maps $\T\to \Z$ and $\T'\to \Z$ to get universal families $\phi:\, \U\to \T$ and $\phi':\, \U'\to \T'$ respectively.

Let $D$ be the classifying space of polarized Hodge structures of
the weight $n$ primitive cohomology of $M$. The period map
$\Phi':\,\mathcal{T}'\rightarrow D$ assigns to each point in
$\mathcal{T}'$ the corresponding Hodge structure of the fiber. The
main result of this paper is the proof of the following global
Torelli theorem:

\begin{thm}
\label{toerllimainintro} The period map $\Phi':\,
\mathcal{T}'\rightarrow D$ is injective.
\end{thm}

The main idea of the proof of the above theorem is the construction of affine structures on the Teichm\"{u}ller space $\mathcal{T}$  and the Hodge metric completion space $\mathcal{T}^H$ of the Torelli space $\mathcal{T}'$.

This paper is organized as follows. In Section \ref{Section period
map} we review the definition of the classifying space of polarized Hodge structures
and briefly describe the construction of the Torelli space
of polarized and marked Calabi-Yau manifolds and its basic
properties. In Section \ref{Section local property} we review the geometry of the
local deformation theory of complex structures, especially in the
case of Calabi-Yau manifolds.

In Section \ref{AffineonT} we construct a holomorphic affine structure on the
Teichm\"{u}ller space. In Section \ref{THm}, we introduce
the Hodge metric completion space of the
Torelli space that we constructed in \cite{LS}. We show that
the Hodge metric completion space $\mathcal{T}^H_{m}$ are holomorphic affine
manifolds. In Section \ref{Proof}, we show that $\mathcal{T}^H_{m}$ is independent on the level $m$ structure, and for simplicity we denote $\mathcal{T}^H_{m}$ by $\mathcal{T}^H$, which contains the Torelli space as an open dense
submanifold. One may refer to \cite{LS} for
more details about $\mathcal{T}^H$. Based on the
completeness and holomorphic affineness of $\mathcal{T}^H$ and the local Torelli theorem, we prove the injectivity of 
the period map $\Phi':\, \mathcal{T}'\rightarrow D$.
In Section \ref{globalSection}, we extend the local canonical sections of the holomorphic $(n,0)$-classes, constructed in Section \ref{Section local property}, to the global canonical sections on the Hodge completion space $\T^{H}$.

In Appendix \ref{preliminary} we review some facts about the structures of the 
period domain $D$ which are known to experts but 
important to our discussions. We mainly follow Section 3 in \cite{schmid1}
to discuss the classifying space $D$,  its compact dual
$\check{D}$, the tangent space of $D$, and the tangent space of
$\check{D}$ as quotients of Lie groups and Lie algebras respectively. Moreover, we describe the relation between the period domain
and the unipotent Lie group $N_+=\exp(\mathfrak{n}_+)$, which is the
corresponding Lie group of the nilpotent Lie algebra. We also
describe the matrix representation of elements in those Lie groups
and Lie algebras if we fix a base point in $D$ and an adapted basis
for the Hodge decomposition of the base point.

Now we briefly recall the history of the Torelli problem. The idea
to study the periods of abelian varieties on Riemann surfaces goes
back to Riemann. In 1914 Torelli asked wether two curves are
isomorphic if they have the same periods. See \cite{Tor} for detail.
In \cite{aw1} Weil reformulated the Torelli problem as follows:
Suppose for two Riemann surfaces, there exists an isomorphism of
their Jacobians which preserves the canonical polarization of the
Jacobians, is it true that the two Riemann surfaces are isomorphic.
Andreotti proved Weil's version of the Torelli problem in
\cite{and}.

Another important achievement about the Torelli problem, conjectured
by Weil in \cite{W2}, was the proof of the global Torelli Theorem
for K3 surfaces, essentially due to Shafarevich and
Piatetski-Shapiro in \cite{PS}. Andreotti's proof is based on
specific geometric properties of Riemann surfaces. The approach of
Shafarevich and Piatetski-Shapiro is based on the arithmeticity of
the mapping class group of a K3 surface. It implies that the special
K3 surfaces,
 the Kummer surfaces, are everywhere dense subset in the moduli of K3 surfaces. Shafarevich and
Piatetski-Shapiro observed that the period map has degree one on the
set of Kummer surfaces, which implies the global Torelli theorem.

The literature about the Torelli problem is enormous. Many authors made
very substantial contributions to the general Torelli problem. We
believe that it is impossible to give a complete list of all the
achievements in this area and its applications. In \cite{ver}
Verbitsky used an approach similar to ours in his proof of the global
Torelli theorem for hyperk\"{a}hler manifolds.

This paper is a revised and shortened version of our joint paper with Andrey Todorov who passed away 
in March 2012. By using methods we recently developed, we correct an error in an earlier version of this paper in proving the existence of affine structure 
on the Teichm\"uller space, and introduce the Torelli space and its Hodge metric completion which are the most natural spaces for studying period maps and the global Torelli problem.  
To avoid overlaps with the other papers we have written recently on this topic, in this paper we only sketch the main ideas of our proofs, and most details are given in \cite{LS}. We sincerely 
apologize for the confusion caused by our careless mistake and the unexpected death of  Professor Andrey Todorov.

We would like to thank Professor S.-T. Yau for his constant interest, for
sharing his ideas and encouragement during the work on this project.
We had numerous useful conversations with X. Sun. We would also like to thank Professors
F. Bogomolov, P. Deligne, Bart van den Dries, R. Friedman, V.
Golyshev, Radu Laza, S. Li, B. Lian, E. Loojienga,  Yu. Manin, Veeravalli S. Varadarajan and M. Verbitsky for
their interest and valuable suggestions.

\section{The Period Map}\label{Section period map}
In this section, we review the definitions and some basic results
about period domain, Teichm\"{u}ller space and Torelli space, as well as the period maps on them. In
Section \ref{section period domain}, we recall the definition and
some basic properties of period domain. In Section \ref{section
construction of Tei}, we discuss the construction of the
Teichm\"ulller space and Torelli space of Calabi-Yau manifolds based on the works of
Popp \cite{Popp}, Viehweg \cite{v1} and Szendroi \cite{sz} on the
moduli spaces of Calabi-Yau manifolds. In Section \ref{section
period map}, we define the period maps from the Teichm\"{u}ller space and Torelli space
to the period domain. We remark that most of the results in this
section are standard and can be found from the literature in the
subjects.

\subsection{Classifying space of polarized Hodge structure}\label{section period
domain}
We first review the construction of the classifying space of
polarized Hodge structures, which is also called the period domain.
We refer the reader to Section 3 of \cite{schmid1} for details.

This paper mainly considers polarized and marked Calabi-Yau manifolds.
A pair $(M,L)$ consisting of a Calabi-Yau manifold $M$ of complex dimension $n$ and an ample
line bundle $L$ over $M$ is called a {polarized Calabi-Yau
manifold}. By abuse of notation the Chern class of $L$ will also be
denoted by $L$ and $L\in H^2(M,\mathbb{Z})$. We use
$h^n=\dim_{\mathbb{C}}H^n(M,\mathbb{C})$ to denote the Betti number.

We fix a lattice $\Lambda$ with a pairing $Q_{0}$, where $\Lambda$ is isomorphic to $H^n(M_{0},\mathbb{Z})/\text{Tor}$ for some Calabi--Yau manifold $M_{0}$ and $Q_{0}$ is defined by the cup-product.
For a polarized Calabi--Yau manifold $(M,L)$, we define a marking $\gamma$ as an isometry of the lattices
\begin{equation}\label{marking}
\gamma :\, (\Lambda, Q_{0})\to (H^n(M,\mathbb{Z})/\text{Tor},Q).
\end{equation}

\begin{definition}
Let the pair $(M,L)$ be a polarized Calabi--Yau manifold, we call the
triple $(M,L,\gamma)$ a polarized and
marked Calabi--Yau manifold.
\end{definition}

For a polarized and marked Calabi-Yau manifold $(M,L,\gamma)$ with background
smooth manifold $X$, the marking $\gamma$ identifies
$H_n(M,\mathbb{Z})/Tor$ isometrically to the fixed lattice $\Lambda$. This
gives us a canonical identification of the middle dimensional de
Rahm cohomology of $M$ to that of the background manifold $X$,
\begin{equation*}
H^n(M)\cong H^n(X)
\end{equation*}
where the coefficient ring can be ${\mathbb{Q}}$, ${\mathbb{R}}$ or
${\mathbb{C}} $.

Since the polarization $L$ is an integer class, it defines a map
\begin{equation*}
L:\, H^n(X,{\mathbb{Q}})\to H^{n+2}(X,{\mathbb{Q}})
\end{equation*}
given by $A\mapsto L\wedge A$ for any $A\in H^n(X,{\mathbb{Q}})$. We
denote by $H_{pr}^n(X)=\ker(L)$ the primitive cohomology groups
where,
again, the coefficient ring is ${\mathbb{Q}}$, ${\mathbb{R}}$ or ${\mathbb{C}%
}$. We let $H_{pr}^{k,n-k}(M)=H^{k,n-k}(M)%
\cap H_{pr}^n(M,{\mathbb{C}})$ and denote its dimension by
$h^{k,n-k}$. We have the Hodge decomposition
\begin{align}  \label{cl10}
H_{pr}^n(M,{\mathbb{C}})=H_{pr}^{n,0}(M)\oplus\cdots\oplus
H_{pr}^{0,n}(M).
\end{align}
It is easy to see that for a polarized Calabi-Yau manifold, since $H^2(M, {\mathcal O}_M)=0$, we have
$$H_{pr}^{n,0}(M)= H^{n,0}(M), \ H_{pr}^{n-1,1}(M)= H^{n-1,1}(M).$$

The Poincar\'e bilinear form $Q$ on $H_{pr}^n(X,{\mathbb{Q}})$ is
defined by
\begin{equation*}
Q(u,v)=(-1)^{\frac{n(n-1)}{2}}\int_X u\wedge v
\end{equation*}
for any $d$-closed $n$-forms $u,v$ on $X$. The bilinear form $Q$ is
symmetric if $n$ is even and is skew-symmetric if $n$ is odd.
Furthermore, $Q $ is non-degenerate and can be extended to
$H_{pr}^n(X,{\mathbb{C}})$ bilinearly, and satisfies the
Hodge-Riemann relations
\begin{eqnarray}  \label{cl30}
Q\left ( H_{pr}^{k,n-k}(M), H_{pr}^{l,n-l}(M)\right )=0\ \
\text{unless}\ \ k+l=n, \quad\text{and}\quad
\end{eqnarray}
\begin{eqnarray}  \label{cl40}
\left (\sqrt{-1}\right )^{2k-n}Q\left ( v,\bar v\right )>0\ \
\text{for}\ \ v\in H_{pr}^{k,n-k}(M)\setminus\{0\}.
\end{eqnarray}

The above Hodge decomposition of $H_{pr}^n(M,{\mathbb{C}})$ can also
be described via the Hodge filtration. Let $f^k=\sum_{i=k}^n
h^{i,n-i}$, and
\begin{equation*}
F^k=F^k(M)=H_{pr}^{n,0}(M)\oplus\cdots\oplus H_{pr}^{k,n-k}(M)
\end{equation*}
from which we have the decreasing filtration
\begin{equation*}
H_{pr}^n(M,{\mathbb{C}})=F^0\supset\cdots\supset F^n.
\end{equation*}

We know that
\begin{eqnarray}  \label{cl45}
\dim_{\mathbb{C}} F^k=f^k,
\end{eqnarray}
\begin{eqnarray}  \label{cl46}
H^n_{pr}(X,{\mathbb{C}})=F^{k}\oplus \bar{F^{n-k+1}},\quad\text{and}\quad
\end{eqnarray}
\begin{eqnarray}  \label{cl48}
H_{pr}^{k,n-k}(M)=F^k\cap\bar{F^{n-k}}.
\end{eqnarray}
In term of the Hodge filtration $F^n\subset\cdots\subset
F^0=H_{pr}^n(M,{\mathbb{C}})$, the Hodge-Riemann relations can be
written as
\begin{eqnarray}  \label{cl50}
Q\left ( F^k,F^{n-k+1}\right )=0, \quad\text{and}\quad
\end{eqnarray}
\begin{eqnarray}  \label{cl60}
Q\left ( Cv,\bar v\right )>0 \ \ \text{if}\ \ v\ne 0,
\end{eqnarray}
where $C$ is the Weil operator given by $Cv=\left (\sqrt{-1}\right
)^{2k-n}v$ when $v\in H_{pr}^{k,n-k}(M)$. The classifying space $D$
for polarized Hodge structures with data \eqref{cl45} is the space
of all such Hodge filtrations
\begin{equation*}
D=\left \{ F^n\subset\cdots\subset F^0=H_{pr}^n(X,{\mathbb{C}})\mid %
\eqref{cl45}, \eqref{cl50} \text{ and } \eqref{cl60} \text{ hold}
\right \}.
\end{equation*}
The compact dual $\check D$ of $D$ is
\begin{equation*}
\check D=\left \{ F^n\subset\cdots\subset F^0=H_{pr}^n(X,{\mathbb{C}})\mid %
\eqref{cl45} \text{ and } \eqref{cl50} \text{ hold} \right \}.
\end{equation*}

The classifying space $D\subset \check D$ is an open set. We note
that the
conditions \eqref{cl45}, \eqref{cl50} and \eqref{cl60} imply the identity %
\eqref{cl46}. From the definition of classifying space we naturally get
the {Hodge bundles} on $\check{D}$ by associating to each point
in $\check{D}$ the vector spaces $\{F^k\}_{k=0}^n$ in the Hodge
filtration of that point. Without confusion we will also denote by $F^k$
the bundle with $F^k$ as fiber, for each $0\leq k\leq n$.

We remark that one may refer to Appendix \ref{preliminary} for more detailed descriptions of the classifying space $D$, its compact dual $\check{D}$, and the tangent space at each point of them, by viewing them as quotients of Lie groups and Lie algebras.

\subsection{Constructions of moduli, Teichm\"uller, and Torelli spaces}\label{section construction of Tei}
In this subsection we briefly describe the construction of the Teichmuller space and Torelli space of Calabi-Yau manifolds and discuss their basic properties.

We need the concept of analytic family (versal or universal) of compact complex
manifolds, we refer to page~8-10 in \cite{su}, page~94 in
\cite{Popp} or page 19 in \cite{v1}, for equivalent definitions and more details. 

Let $(M,L)$ be a polarized Calabi--Yau manifold.
Recall that a marking of $(M,L)$ is defined as an isometry
$$\gamma :\, (\Lambda, Q_{0})\to (H^n(M,\mathbb{Z})/\text{Tor},Q).$$
For any integer $m\geq 3$, we follow the definition of Szendr\"oi \cite{sz}
 to define an $m$-equivalent relation of two markings on $(M,L)$ by 
$$\gamma\sim_{m} \gamma' \text{ if and only if } \gamma'\circ \gamma^{-1}-\text{Id}\in m \cdot\text{End}(H^n(M,\mathbb{Z})/\text{Tor}),$$
and denote $[\gamma]_{m}$ to be the set of all the equivalent classes of $\gamma$.
Then we call $[\gamma]_{m}$ a level $m$
structure on the polarized Calabi--Yau manifold.
For moduli space of
polarized Calabi-Yau manifold with level $m$ structure, we have the
following theorem, which is a reformulation of Theorem~2.2 in
\cite{sz}, we just take the statement we need in this paper. One can
also look at \cite{Popp} and \cite{v1} for more details about the
construction of moduli spaces of Calabi-Yau manifolds.

\begin{thm}\label{Szendroi theorem 2.2}
Let $m\geq 3$ and $M$ be polarized Calabi-Yau manifold with level
$m$ structure, then there exists a quasi-projective complex manifold
$\Z$ with a universal family of Calabi-Yau manifolds,
\begin{align}\label{szendroi versal family}
\mathcal{X}_{\Z}\rightarrow \Z,
\end{align}
containing $M$ as a fiber, and polarized by an ample line bundle
$\mathcal{L}_{\Z}$ on $\mathcal{X}_{\Z}$.
\end{thm}

Let $(M,L)$ be a polarized Calabi-Yau manifold. We define $\T$
to be the universal cover of $\Z$ with the covering map
\begin{align*}
\pi_m:\, \T\rightarrow \Z
\end{align*}
and the family
\begin{align*}
\U\rightarrow \T
\end{align*}
to be the pull-back of the family (\ref{szendroi versal family}) by
the covering map $\pi_{m}$. 
\begin{proposition}\label{imp}
The \tei space $\mathcal{T}$ is a simply connected smooth complex
manifold, and the family
\begin{align}\label{versal family over Teich}
\U\rightarrow\mathcal{T}
\end{align}
containing $M$ as a fiber, is a universal family.
\end{proposition}
\begin{proof}
For the first part, because $\Z$ is smooth complex manifold, the
universal cover of $\Z$ is simply connected smooth complex
manifold. For the second part, we know that the family
(\ref{szendroi versal family}) is universal family at each point of
$\Z$, and $\pi$ is locally bi-holomorphic, thus the pull-back
family via $\pi$ is also universal, as the universal property of analytic families is  local.
\end{proof}
Note that the \tei space $\mathcal{T}$ does not depend on the choice
of $m$. In fact let $m_1$, $m_2$ be two different integers,
$\U_1\rightarrow\T_1$ and $\U_2\rightarrow\T_2$ are two versal
families constructed via level $m_1$ and level $m_2$ respectively as
above, both of which contain $M$ as a fiber. By using the fact that
$\T_1$ and $\T_2$ are simply connected and the definition of universal
family, we have a bi-holomorphic map $f:\T_1\rightarrow\T_2$, such
that the universal family $\U_1\rightarrow\T_1$ is the pull back of the
versal family $\U_2\rightarrow \T_2$ by $f$. Thus these two families
are isomorphic to each other.

Now we introduce the definition of  \textit{Torelli space} $\mathcal{T}'$. For this  we need to consider the  equivalence classes of triples $(M, L, \gamma)$ which is called a polarized and marked Calabi-Yau manifold with polarization $L$ and marking $\gamma$. To be more precise, 
 two triples $(M, L, \gamma)$ and $(M', L', \gamma'\})$ is called equivalent, if there exists a bi-holomorphic map $f:\,M\to M'$ with
\begin{align*}
f^*L'&=L,\quad \quad
f^*\gamma'=\gamma,
\end{align*} 
where $f^*\gamma'$ is given by $\gamma':\, (\Lambda, Q_{0})\to (H^n(M',\mathbb{Z})/\text{Tor},Q)$ composed with $$f^*:\, (H^n(M',\mathbb{Z})/\text{Tor},Q)\to (H^n(M,\mathbb{Z})/\text{Tor},Q),$$
then we define $[M, L, \gamma]=[M', L',\gamma']\in \mathcal{T}'$.  More precisely we define the Torelli space as follows.

\begin{definition}
The Torelli space $\T'$ of Calabi--Yau manifolds is the connected component of the moduli space of equivalence classes of polarized and marked Calabi--Yau manifolds, which contains $(M, L)$.
\end{definition}

There is a natural covering map $\pi_m':\,\mathcal{T}'\to\mathcal{Z}_m$ by mapping $(M, L, \gamma)$ to $(M, L, [\gamma]_{m})$. From this we see easily that $\T'$ is a smooth and connected complex manifold.
We can also get a pull-back universal family $\phi':\, \U'\to \T'$ on the Torelli space $\T'$ via the covering map $\pi_m'$.

Recall that the Teichm\"uller space $\T$ is defined to be the universal covering space of $\mathcal{Z}_{m}$ with covering map $\pi_{m}:\, \T\to \mathcal{Z}_{m}$. Then we can lift $\pi_{m}$ via the covering map $\pi_m':\,\mathcal{T}'\to\mathcal{Z}_m$ to get a covering map $\pi:\, \T\to \T'$, such that the following diagram commutes.
\begin{equation}\label{cover1}
\xymatrix{
\T \ar[dr]^-{\pi} \ar[dd]^-{\pi_m} &\\
&\T'\ar[dl]^-{\pi_m'}\\
\mathcal{Z}_{m} &\ .}
\end{equation}

We remark that our method of proving the global Torelli theorem for
polarized and marked Calabi-Yau manifold applies without change to
the case of more general projective manifolds with
trivial canonical line bundle, including polarized and marked
hyperk\"{a}hler manifolds and K3 surfaces. In these cases the
Teichm\"{u}ller spaces can be embedded inside symmetric domains of non-compact
type, so naturally have global holomorphic affine flat coodinates.


\subsection{The period map}\label{section period map}
We are now ready to define the period maps from the the Teichm\"uller space and Torelli space to the
period domain. 

Let us denote the period map on the moduli space by $\Phi_{\mathcal{Z}_m}: \,\mathcal{Z}_{m}\rightarrow D/\Gamma$, where 
$\Gamma$ denotes the global monodromy group which acts
properly and discontinuously on the period domain $D$.
Since the period map $\Phi_{\mathcal{Z}_m}$ is locally liftable, we have the lifted period map $\Phi:\, \T \to D$ on the universal cover $\T$ of $\Z$ such the following diagram commutes
$$\xymatrix{
\T \ar[r]^-{\Phi} \ar[d]^-{\pi_m} & D\ar[d]^-{\pi}\\
\mathcal{Z}_m \ar[r]^-{\Phi_{\mathcal{Z}_m}} & D/\Gamma
}$$

For any point $q\in\T'$, let $(M_q,L_{q},\gamma_{q})$ be the fiber of
family $\pi:\, \U'\rightarrow \T'$, which is a polarized and marked
Calabi-Yau manifold. Since on the Torelli space we have fixed lattice $(\Lambda_{0},Q_{0})$, we can use this to identify $(H^n(M_q, \
\mathbb{C}),Q)$ isometrically for the fiber of any point $q$ in $\T'$, and thus get a canonical trivial
bundle $H^n(M_p,\mathbb{C})\times \T'$ with $p\in T'$ the fixed point. We have similar
identifications for $H^n(M,\mathbb{Q})$ and $H^n(M,\mathbb{Z})$.

The period map from $\T'$ to $D$ is defined by
assigning each point $q\in\T'$ the Hodge structure on $M_p$,
\begin{align*}
\Phi':\, \T'\rightarrow D
\end{align*}
with $\Phi'(q)=\gamma_{q}^{-1}(\{F^n(M_q)\subset\cdots\subset F^0(M_q)\})$ for any $q\in \T'$, where $(M_{q},L_{q},\gamma_{q})$ is the fiber over $q$. We denote
$F^k(M_q)$ by $F^k_q$ for convenience.

In a summary we have defined the period maps on the Teichm\"uller space and Torelli space. Moreover they fit into the following commutative diagram
\begin{align}\label{periods}
\xymatrix{
\T \ar[dr]^-{\pi}\ar[rr]^-{\Phi} \ar[dd]^-{\pi_m} && D\ar[dd]^-{\pi_{D}}\\
&\T'\ar[ur]^-{\Phi'}\ar[dl]_-{\pi'_{m}}&\\
\Z \ar[rr]^-{\Phi_{\Z}} && D/\Gamma,
}
\end{align}
where the  maps $\pi_{m}$, $\pi'_{m}$ and $\pi$ are all natural covering maps between the corresponding spaces as in \eqref{cover1}.

Period map has several good properties, we refer the reader to
Chapter~10 in \cite{Voisin} for details. Among them the most
important is the following Griffiths transversality. Let $\Phi:\, \T \to D$ be
the period map on the Teichm\"uller space, then $\Phi$ is holomorphic map and for any $q\in\T$
and $v\in T_q^{1,0}\T$, the tangent map satisfies,
\begin{equation}\label{griffiths transversality quotient version}
\Phi_*(v)\in \bigoplus_{k=1}^{n}
\text{Hom}\left(F^k_q/F^{k+1}_q,F^{k-1}_q/F^{k}_q\right),
\end{equation}
where $F^{n+1}=0$, or equivalently
\begin{equation*}
\Phi_*(v)\in \bigoplus_{k=0}^{n} \text{Hom} (F^k_q,F^{k-1}_q/F^k_q).
\end{equation*}
This is called the Griffiths transversality.

In \cite{GS}, Griffiths and Schmid studied the  {Hodge metric} on the period domain $D$. We denote it by $h$. In particular, this Hodge metric is a complete homogeneous metric. 
By local Torelli theorem for Calabi--Yau manifolds, we know that the period maps $\Phi_{\mathcal{Z}_m}$, $\Phi$ and $\Phi'$ are locally injective. Thus it follows from \cite{GS} that the pull-back of $h$ by $\Phi_{\mathcal{Z}_m}$, $\Phi$ and $\Phi'$ on $\mathcal{Z}_m$, $\mathcal{T}$ and $\T'$ respectively are both well-defined K\"ahler metrics. By abuse of notation, we still call these pull-back metrics the \textit{Hodge metrics}, and they are all denoted by $h$. 

Hodge bundles over $\mathcal{T}$ are the pull-back of the Hodge
bundles over $D$ through the period map. For convenience, we still
denote them by $F^k$, for each $0\leq k\leq n$. We will also denote
by $P_p^{k}$ the projection from $H^n(M,\mathbb{C})$ to $F^k_p$ with
respect to the Hodge filtration at $M_p$, and $P_p^{n-k,k}$ the
projection from $H^n(M,\mathbb{C})$ to $H^{n-k,k}(M_p)$ with respect
to the Hodge decomposition at $M_p$.

With all of these preparations, we are ready to state precisely
the main theorem of this paper.
\begin{thm}\label{thm}
The period map $\Phi'$ constructed above is an embedding,
\begin{align*}
\Phi': \, \T' \hookrightarrow D.
\end{align*}
\end{thm}
We will outline the main ideas of proofs of the above theorem, and refer the reader to \cite{LS} for details of the arguments and more results.  
\section{Local Geometric Structure of the Teichm\"{u}ller Space}\label{Section local property}
In this section, we review the local deformation theory of polarized
Calabi-Yau manifolds, which will be needed for the construction of
the global holomorphic affine structure and global holomorhic sections of the Hodge bundle $F^n$ on the Teichm\"{u}ller space
in Section \ref{globalSection}. This section can be skipped if the reader is only interested in the proof of the Theorem \ref{thm}.  

In Section \ref{section
local deformation of complex structure}, we briefly review the basic
local deformation theory of complex structures. In Section
\ref{section local  deformation of cy mfds}, we recall the local Kuranishi
deformation theory of Calabi-Yau manifolds, which depends on the Calabi-Yau metric in a substantial way. In Section
\ref{section local canonical section}, we describe a local family of the canonical
holomorphic $(n,0)$-forms as a section of the Hodge bundle $F^n$ over
the local deformation space of Calabi-Yau manifolds, from which we obtain an expansion
of the family of holomorphic $(n,0)$-classes as given in Theorem \ref{expcoh}.

Most of the results in this section are standard now in the
literatures, and can be found in \cite{km1}, \cite{tod1}, and
\cite{tian1}. For reader's convenience, we also briefly review some
arguments. We remark that one may use a more algebraic approach to
Theorem \ref{expcoh} by using the local Torelli theorem and the
Griffiths transversality.

\subsection{Local deformation of complex structure}\label{section local deformation of complex structure}
Let $X$ be a smooth manifold of dimension $\dim_{\mathbb{R}} X=2n$
and let $J$ be an integrable complex structure on $X$. We denote by
$M=(X,J)$ the corresponding complex manifold, and $\partial$,
$\bar{\partial}$ the corresponding differential operators on $M$.

Let $\phi\in A^{0,1}\left (M,T^{1,0}M\right )$ be a $T^{1,0}M$-
valued smooth $(0,1)$-form. For any point $x\in M$, and any local
holomorphic coordinate chart $(U, z_1,\cdots,z_n)$ around $x$. Let
us express
$\phi=\phi^i_{\bar{j}}d\bar{z}_j\otimes\frac{\partial}{\partial
z_i}=\phi^i\partial_i$, where $\phi^i=\phi^i_{\bar{j}}d\bar{z}_j$
and $\partial_i=\frac{\partial}{\partial z_i}$ for simplicity. Here we use the standard convention to sum over the
repeated indices. We
can view $\phi$ as a map
\begin{align*}
\phi:\, \Omega^{1,0}(M)\to \Omega^{0,1}(M)
\end{align*}
such that locally we have
\begin{align*}
\phi(dz_i)=\phi^i\ \ \ \ \ \text{for}\ \  1\leq i\leq n.
\end{align*}

We use $\phi$ to describe deformation of complex structures. Let
\begin{align*}
\Omega_\phi^{1,0}(x)=\text{span}_\C\{ dz_1+\phi(dz_1), \cdots,
dz_n+\phi(dz_n)\},\quad\text{and}\quad
\end{align*}
\begin{align*}
\Omega_\phi^{0,1}(x)=\text{span}_\C\{ d\bar z_1+\bar\phi(d\bar z_1),
\cdots, d\bar z_n+\bar\phi(d\bar z_n)\},
\end{align*}
if $\Omega_\phi^{1,0}(x)\cap\Omega_\phi^{0,1}(x)=0$ for any $x$,
then we can define a new almost complex structure $J_\phi$ by letting
$\Omega_\phi^{1,0}(x)$ and $\Omega_\phi^{0,1}(x)$ be the eigenspaces
of $J_\phi(x)$ with respect to the eigenvalues $\sqrt{-1}$ and
$-\sqrt{-1}$ respectively, and we call such $\phi$ a
Beltrami differential.

It was proved in \cite{nn} that the almost complex structure
$J_\phi$ is integrable if and only if
\begin{eqnarray}\label{int}
\bar\pa\phi=\frac{1}{2}[\phi,\phi].
\end{eqnarray}

If (\ref{int}) holds, we will call $\phi$ an integrable Beltrami
differential and denote by $M_\phi$ the corresponding complex
manifold. Please see Chapter 4 in \cite{km1} for more details about the
deformation of complex structures.

Let us recall the notation for contractions and Lie bracket of Beltrami
differentials. Let $(U,z_1, \cdots,z_n)$ be the local coordinate
chart defined above, and $\Omega=fdz_1\wedge\cdots\wedge dz_n$ be a
smooth $(n,0)$-form on $M$, and $\phi\in A^{0,1}(M,T^{1,0}M)$ be a
Beltrami differential. We define
\begin{align*}
\phi\lrcorner\Omega=\sum\limits_{i}(-1)^{i-1}f\phi^i\wedge
dz_1\wedge\cdots\wedge\widehat{dz_i}\wedge\cdots\wedge dz_n.
\end{align*}
For Beltrami differentials $\phi$, $\psi\in A^{0,1}(M,T^{1,0}M)$,
with $\phi=\phi^i\partial_i$ and $\psi=\psi^k\partial_k$, recall that the Lie bracket is defined as
\begin{align*}
[\phi,\psi]=\sum\limits_{i,k}\left(\phi^i\wedge\partial_i\psi^k+\psi^i\wedge\partial_i\phi^k\right)\otimes\partial_k,
\end{align*}
where $\partial_i\phi^k=\frac{\partial\phi^k_{\bar{l}}}{\partial
z_i}d\bar{z}_l$ and
$\partial_i\psi^k=\frac{\partial\psi^k_{\bar{l}}}{\partial
z_i}d\bar{z}_l$.

For $k$ Beltrami differentials $\phi_1,\cdots,\phi_k\in
A^{0,1}(M,T^{1,0}M)$, with $\phi_\alpha=\phi^i_\alpha\partial_i$ and $1\leq\alpha\leq k$, we define
\begin{align*}
\phi_1\wedge\cdots\wedge\phi_k=\sum_{i_1<\cdots<i_k}\left (
\sum_{\sigma\in
S_k}\phi_{\sigma(1)}^{i_1}\wedge\cdots\wedge\phi_{\sigma(k)}^{i_k}
\right )\otimes \left (\partial_{i_1}\wedge\cdots\wedge
\partial_{i_k}\right),
\end{align*}
where $S_k$ is the symmetric group of $k$ elements. Especially we
have
\begin{align*}
\wedge^k\phi=k!\sum_{i_1<\cdots<i_k}\left (
\phi^{i_1}\wedge\cdots\wedge\phi^{i_k}\right )\otimes \left
(\partial_{i_1}\wedge\cdots\wedge \partial_{i_k}\right ).
\end{align*}

Then we define the contraction,
\begin{align*}
&(\phi_1\wedge\cdots\wedge\phi_k)\lrcorner\Omega=
\phi_1\lrcorner(\phi_2\lrcorner(\cdots\lrcorner(\phi_k\lrcorner\Omega))) \\
=& \sum_{I=(i_1,\cdots,i_k)\in A_k}(-1)^{|I|+\frac{(k-1)(k-2)}{2}-1}
f\left ( \sum_{\sigma\in
S_k}\phi_{\sigma(1)}^{i_1}\wedge\cdots\wedge\phi_{\sigma(k)}^{i_k}
\right )\wedge dz_{I^c},
\end{align*}
where $A_k$ is the index set
\begin{equation*}
A_k=\{ (i_1,\cdots,i_k)\mid 1\leq i_1<\cdots<i_k\leq n\}.
\end{equation*}
Here for each $I=(i_1,\cdots,i_k)\in A_k$, we let $|I|=i_1+\cdots+i_k$ and $%
dz_{I^c}=dz_{j_1}\wedge\cdots\wedge dz_{j_{n-k}}$ where
$j_1<\cdots<j_{n-k}$ and $j_\alpha\ne i_\beta$ for any
$\alpha,\beta$. With the above notations, for any Beltrami
differentials $\phi$, $\psi\in A^{0,1}(M,T^{1,0}M)$ one has the
following identity which was proved in \cite{tian1}, \cite{tod1},
\begin{align}\label{tian tod lemma}
\partial((\phi\wedge\psi)\lrcorner\Omega)=-[\phi,\psi]\lrcorner\Omega+\phi\lrcorner\partial(\psi\lrcorner\Omega)+\psi\lrcorner\partial(\phi\Omega).
\end{align}

The following notation will be needed
in the construction of the local canonical family of holomorphic $(n,0)$-classes.
\begin{align}\label{family of smooth n0 form}
e^{\phi}\lrcorner\Omega=\sum\limits_{k\geq
0}\frac{1}{k!}\wedge^k\phi\lrcorner\Omega.
\end{align}

By direct computation, we see that
$e^{\phi}\lrcorner\Omega=f\left(dz_1+\phi(dz_1)\right)\wedge\cdots\wedge\left(dz_n+\phi(dz_n)\right)$
is a smooth $(n,0)$-form on $M_{\phi}$.

\subsection{Local deformation of Calabi-Yau manifold}\label{section local deformation of cy mfds}
 For a point $p\in\mathcal{T}$, we denote by $(M_p, L)$ the corresponding polarized and marked Calabi-Yau manifold
as the fiber over $p$. Yau's solution of the Calabi conjecture
implies that there exists a unique Calabi-Yau metric $h_p$ on $M_p$, and the
imaginary part $\omega_p=$Im\,$h_p\in L$ is the corresponding
K\"{a}hler form. First by using the Calabi-Yau metric we have the following lemma,
\begin{lemma}\label{iso}
Let $\Omega_p$ be a nowhere vanishing holomorphic $(n,0)$-form on
$M_p$ such that
\begin{eqnarray}\label{normalization}
\left ( \frac{\sqrt{-1}}{2}\right
)^n(-1)^{\frac{n(n-1)}{2}}\Omega_p\wedge\bar{\Omega_p}=\omega_p^n.
\end{eqnarray}
Then the map $\iota:\, A^{0,1}\left ( M, T^{1,0}M\right )\to
A^{n-1,1}(M)$ given by $\iota(\phi)=\phi\lrcorner\Omega_p$ is an
isometry with respect to the natural Hermitian inner product on both
spaces induced by $\omega_p$. Furthermore, $\iota$ preserves the
Hodge decomposition.
\end{lemma}

Let us briefly recall the proof. We can pick local coordinates
$z_1,\cdots,z_n$ on $M$ such that $\Omega_p=dz_1\wedge\cdots\wedge
dz_n$ locally and $\omega_p=\frac{\sqrt{-1}}{2} g_{i\bar
j}dz_i\wedge d\bar z_j$, then the condition \eqref{normalization}
implies that $\det[g_{i\bar j}]=1$. The lemma follows from direct
computations.

Let $\partial_{M_p}$, $\bar{\partial}_{M_p}$,
$\bar{\partial}_{M_p}^*$, $\square_{M_p}$, $G_{M_p}$, and
$\mathbb{H}_{M_p}$ be the corresponding operators in the Hodge
theory on $M_p$, where $\bar{\partial}_{M_p}^*$ is the adjoint
operator of $\bar{\partial}_{M_p}$, $\square_{M_p}$ the Laplace
operator, and $G_{M_p}$ the corresponding Green operator. We let
$\mathbb{H}_{M_p}$denote the harmonic projection onto the kernel of
$\square_{M_p}$. We also denote by $\mathbb{H}^{p,q}(M_p,\, E)$ the
harmonic $(p,q)$-forms with value in a holomorphic vector bundle $E$
on $M_p$.

By using the Calabi-Yau metric we have a more precise description of
the local deformation of a polarized Calabi-Yau manifold. First from
Hodge theory, we have the following identification
\begin{align*}
T_p^{1,0}\T\cong \h^{0,1}\left ( M_p,T_{M_p}^{1,0}\right ).
\end{align*}

From Kuranishi theory we have the following local convergent power
series expansion of the Beltrami differentials, which is now well-known as the
Bogomolov-Tian-Todorov theorem.
\begin{thm}\label{flatcoord}
Let $\phi_1,\cdots,\phi_N \in \h^{0,1}\left (
M_p,T_{M_p}^{1,0}\right )$ be a basis. Then there is a unique power
series
\begin{eqnarray}\label{10}
\phi(\tau)=\sum_{i=1}^N \tau_i\phi_i +\sum_{|I|\geq 2}\tau^I\phi_I
\end{eqnarray}
which converges for $|\tau|<\eps$ small. Here $I=(i_1,\cdots,i_N)$ is a
multi-index, $\tau^I=\tau_1^{i_1}\cdots\tau_N^{i_N}$ and $\phi_I\in
A^{0,1}\left ( M_p,T_{M_p}^{1,0}\right )$. Furthermore, the family of
Beltrami differentials $\phi(\tau)$ satisfy the following
conditions:
\begin{align}\label{charflat}
\begin{split}
&\bar\pa_{M_p}\phi(\tau)=\frac{1}{2}[\phi(\tau),\phi(\tau)],\\
&\bar\pa_{M_p}^*\phi(\tau)=0,\\
&\phi_I\lrcorner\Omega_p =\pa_{M_p}\psi_I,
\end{split}
\end{align}
for each $|I|\geq 2$ where $\psi_I\in A^{n-2,1}(M_p)$ are smooth $(n-2,1)$-forms.
By shrinking $\eps$ we can pick each $\psi_I$ appropriately such
that $\sum_{|I|\geq 2}\tau^I\psi_I$ converges for $|\tau|<\eps$.
\end{thm}

\begin{remark}\label{flatcoordremark}
The coordinate $\{\tau_1, \cdots, \tau_N\}$ depends on the choice of basis $\phi_1,\cdots,\phi_N \in \h^{0,1}\left (
M_p,T_{M_p}^{1,0}\right )$. But one can also determine the coordinate by  fixing a basis $\{\eta_0\}$ and $\{\eta_1, \cdots, \eta_N\}$ for $H^{n,0}(M_p)$ and $H^{n-1,1}(M_p)$ respectively. In fact, Lemma \ref{iso} implies that there is a unique choice of $\phi_1, \cdots, \phi_N$ such that $\eta_k=[\phi_k\lrcorner \eta_0],$ for each $1\leq k\leq N$.
\end{remark}

Theorem \ref{flatcoord} was proved in \cite{tod1}, and in \cite{tian1} in a
form without specifying the Kuranishi gauge, the second and the
third condition in (\ref{charflat}). This theorem implies that the local
deformation of a Calabi-Yau manifold is unobstructed. Here we only
mention two important points of its proof. For the convergence of
$\sum_{|I|\geq 2}\tau^I\psi_I$, noting that
$\partial_{M_p}\psi_I=\phi_I\lrcorner\Omega_p$ and
$\bar{\partial}\phi_I\lrcorner\Omega_p=0$, we can pick
$\psi_I=\partial^*_{M_p}G(\phi_I\lrcorner\Omega_p)$. It follows that
\begin{align*}
\|\psi_I\|_{k,\alpha}\leq
C(k,\alpha)\|\phi_I\lrcorner\Omega_p\|_{k-1,\alpha}\leq
C'(k,\alpha)\|\phi_I\|_{k-1,\alpha}.
\end{align*}
The desired convergence follows from the estimate on $\phi_I$. We
note that the convergence of (\ref{10}) follows from standard
elliptic estimate. See \cite{tod1}, or Chapter 4 of \cite{km1} for
details.

For the third condition in (\ref{charflat}), by using the first two
conditions in (\ref{charflat}), for example we have in the case of
$|I|=2$,
\begin{align}\label{firsttwocond}
\overline{\partial}_{M_p}\phi_{ij}=[\phi_i,\phi_j]\ \text{and}\
\overline{\partial}^*_{M_p}\phi_{ij}=0.
\end{align}
Then by using formula (\ref{tian tod lemma}) and Lemma \ref{iso},
we get that
\begin{align*}
[\phi_i,\phi_j]\lrcorner\Omega_p=\partial_{M_p}(\phi_i\wedge\phi_j\lrcorner\Omega_p)
\end{align*}
is $\partial_{M_p}$ exact. It follows that
$\overline{\partial}_{M_p}(\phi_{ij}\lrcorner\Omega_p)=\left(\overline{\partial}_{M_p}\phi_{ij}\right)\lrcorner\Omega_p$ is also
$\partial_{M_p}$ exact. Then by the $\partial\overline{\partial}$-lemma
we have
\begin{align*}
\overline{\partial}_{M_p}(\phi_{ij}\lrcorner\Omega_p)=\overline{\partial}_{M_p}\partial_{M_p}\psi_{ij}
\end{align*}
for some $\psi_{ij}\in A^{n-2,1}$. It follows that
\begin{align*}
\phi_{ij}\lrcorner\Omega_p=\partial_{M_p}\psi_{ij}+\bar{\partial}_{M_p}\alpha+\beta
\end{align*}
for some $\alpha\in A^{n-1,0}(M_p)$ and $\beta\in \mathbb{H}^{n-1,1}(M_p)$. By using the condition
$\overline{\partial}^*_{M_p}\phi_{ij}=0$ and Lemma \ref{iso}, we have
\begin{align*}
\phi_{ij}\lrcorner\Omega_p=\partial_{M_p}\psi_{ij}+\beta.
\end{align*}
 Because $\phi_{ij}$ is not uniquely
determined by condition (\ref{firsttwocond}), we can choose
$\phi_{ij}$ such that the harmonic projection
$\mathbb{H}(\phi_{ij})=0$. Then by using Lemma \ref{iso} again, we
have
\begin{align*}
\phi_{ij}\lrcorner\Omega_p=\partial_{M_p}\psi_{ij}.
\end{align*}
Thus there exists a unique $\phi_{ij}$ which satisfies all three
conditions in (\ref{charflat}). We can then proceed by induction and
the same argument as above to show that the third condition in
(\ref{charflat}) holds for all $|I|\geq 2$. See \cite{tod1} and
\cite{tian1} for details.

Theorem \ref{flatcoord} can be used to  define the local
holomorphic affine flat coordinates $\{\tau_1,\cdots,\tau_N\}$
around $p$, for a given orthonormal basis $\phi_1,\cdots,\phi_N \in
\h^{0,1}\left ( M_p,T_{M_p}^{1,0}\right )$. Sometimes we also denote by
$M_\tau$ the deformation given by the Beltrami differential
$\phi(\tau)$. This local affine coordinates can be extended to a global one 
as discussed in Section \ref{globalSection}.

\subsection{Local canonical section of holomorphic $(n,0)$-classes}\label{section local canonical section}
By using the local deformation theory, in \cite{tod1} Todorov
constructed a canonical local holomorphic section of the line bundle
$H^{n,0}=F^n$ over the local deformation space of a Calabi-Yau manifold
at the differential form level. We first recall the construction of the holomorphic
$(n,0)$-forms in \cite{tod1}.

Let $\phi\in A^{0,1}\left (
M,T^{1,0}M\right )$ be an integrable Beltrami differential and let
$M_\phi$ be the Calabi-Yau manifold defined by $\phi$. We refer the
reader to Section \ref{section local deformation of complex
structure} for the definition of the contraction
$e^\phi\lrcorner\Omega_p$.
\begin{lemma}\label{constructn0}
Let $\Omega_p$ be a nowhere vanishing holomorphic $(n,0)$-form on
$M_p$ and $\{z_1,\cdots,z_n\}$ is a local holomorphic coordinate system with
respect to $J$ such that
\begin{align*}
\Omega_p=dz_1\wedge\cdots\wedge dz_n
\end{align*}
locally. Then the smooth $(n,0)$-form
\begin{align*}
\Omega_\phi=e^{\phi}\lrcorner\Omega_p
\end{align*}
is holomorphic with respect to the complex structure on $M_\phi$ if
and only if $\partial_{M_p}(\phi\lrcorner\Omega_p)=0$.
\end{lemma}

\begin{proof} The proof in \cite{tod1} is by direct computations, here
we give a simple proof.

Being an $(n,0)$-form on $M_\phi$, $e^\phi\lrcorner\Omega_p$ is
holomorphic on $M_\phi$ if and only if
$d(e^\phi\lrcorner\Omega_p)=0$. For any smooth
$(n,0)$-form $\Omega_p$ and Beltrami differential $\phi\in
A^{0,1}\left(M,T^{1,0}M\right)$,
we have the following formula,
\begin{align*}
d(e^{\phi}\lrcorner\Omega_p)=e^{\phi}\lrcorner(\bar{\partial}_{M_p}\Omega_p+\partial_{M_p}(\phi\lrcorner\Omega_p))+(\bar{\partial}_{M_p}\phi-\frac{1}{2}[\phi,\phi])\lrcorner(e^\phi\lrcorner\Omega_p),
\end{align*}
which can be verified by direct computations. In our case the
Beltrami differential $\phi$ is integrable, i.e.
$\bar{\partial}_{M_p}\phi-\frac{1}{2}[\phi,\phi]=0$ and $\Omega_p$
is holomorphic on $M_p$. Therefore we have
\begin{align*}
d(e^{\phi}\lrcorner\Omega_p)=e^{\phi}\lrcorner(\partial_{M_p}(\phi\lrcorner\Omega_p)),
\end{align*}
which implies that $e^{\phi}\lrcorner\Omega_p$ is holomorphic on
$M_{\phi}$ if and only if $\partial_{M_p}(\phi\lrcorner\Omega_p)=0$.
\end{proof}

Now we can construct the canonical family of holomorphic
$(n,0)$-forms on the local deformation space of Calabi-Yau manifolds.

\begin{proposition}\label{canonicalfamily}
We fix on $M_p$ a nowhere vanishing holomorphic $(n,0)$-form
$\Omega_p$ and an orthonormal basis $\{\phi_i\}_{i=1}^N$ of
$\mathbb{H}^1(M_p,T^{1,0}M_p)$. Let $\phi(\tau)$ be the family of
Beltrami differentials given by \eqref{charflat} that defines a local deformation of $M_p$ which we denote by
$M_\tau$. Let
\begin{align}\label{can10}
\Omega^c_p(\tau)=e^{\phi(\tau)}\lrcorner\Omega_p.
\end{align}
Then $\Omega^c_p(\tau)$ is a well-defined nowhere vanishing
holomorphic $(n,0)$-form on $M_\tau$ and depends on $\tau$
holomorphically.
\end{proposition}

\begin{proof}  We call such family the canonical family of
holomorphic $(n,0)$-forms on the local deformation space of $M_p$.
 The fact that $\Omega(\tau)^c$ is a nowhere vanishing
holomorphic $(n,0)$-form on the fiber $M_\tau$ follows from its
definition and Lemma \ref{constructn0} directly. In fact we only
need to check that $\partial_{M_p}(\phi(\tau)\lrcorner\Omega_p)=0$.
By formulae \eqref{10} and \eqref{charflat} we know that
\begin{align*}
\phi(\tau)\lrcorner\Omega_p=\sum_{i=1}^N
\tau_i(\phi_i\lrcorner\Omega_p)+\sum_{|I|\geq 2}\tau^I
(\phi_I\lrcorner\Omega_p)= \sum_{i=1}^N
\tau_i(\phi_i\lrcorner\Omega_p)+\partial_{M_p}\left ( \sum_{|I|\geq
2}\tau^I\psi_I\right ).
\end{align*}

Because each $\phi_i$ is harmonic, by Lemma \ref{iso} we know that
$\phi_i\lrcorner\Omega_p$ is also harmonic and thus
$\partial_{M_p}(\phi_i\lrcorner\Omega_p)=0$. Furthermore, since
$\sum_{|I|\geq 2}\tau^I\psi_I$ converges when $|\tau|$ is small, we
see that $\partial_{M_p}(\phi(\tau)\lrcorner\Omega_p)=0$ from
formula \eqref{charflat}. The holomorphic dependence of
$\Omega^c_p(\tau)$ on $\tau$ follows from formula \eqref{can10} and
the fact that $\phi(\tau)$ depends on $\tau$ holomorphically.
\end{proof}

From Theorem \ref{flatcoord} and Proposition \ref{canonicalfamily} we get the
expansion of the deRham cohomology classes of the canonical family
of holomorphic $(n,0)$-forms. This expansion is closely related to the construction of the holomorphic affine structure on the
Teichm\"{u}ller space. We remark that one may also directly deduce this expansion from
the local Torelli theorem for Calabi-Yau manifold and the Griffiths transversality.
\begin{theorem}\label{expcoh}
Let $\Omega^c_p(\tau)$ be a canonical family defined by
\eqref{can10}.
Then we have the following expansion for $|\tau|<\epsilon$ small,
\begin{eqnarray}\label{cohexp10}
[\Omega^c_p(\tau)]=[\Omega_p]+\sum_{i=1}^N
\tau_i[\phi_i\lrcorner\Omega_p]+A(\tau),
\end{eqnarray}
where $\{[\phi_1\lrcorner\Omega_p], \cdots, [\phi_N\lrcorner\Omega_p]\}$ give a basis of $H^{n-1,1}(M_p)$ and $A(\tau)=O(|\tau|^2)\in \bigoplus_{k=2}^n H^{n-k,k}(M_p)$ denotes terms of
order at least $2$ in $\tau$.
\end{theorem}

\begin{proof} By Theorem \ref{flatcoord} and Proposition \ref{canonicalfamily}
we have
\begin{align}\label{20000}
\begin{split}
\Omega^c_p(\tau)=& \Omega_p+\sum_{i=1}^N
\tau_i(\phi_i\lrcorner\Omega_p)+\sum_{|I|\geq
2}\tau^I(\phi_I\lrcorner\Omega_p)+\sum_{k\geq 2} \frac{1}{k!}
\left ( \wedge^k\phi(\tau)\lrcorner\Omega_p \right )\\
=&\Omega_p+\sum_{i=1}^N
\tau_i(\phi_i\lrcorner\Omega_p)+\pa_{M_p}\left ( \sum_{|I|\geq
2}\tau^I\psi_I\right ) +a(\tau),
\end{split}
\end{align}
where
\begin{align}\label{a(tau)}
a(\tau)=\sum_{k\geq 2} \frac{1}{k!} \left (
\wedge^k\phi(\tau)\lrcorner\Omega_p \right )\in \bigoplus_{k\geq
2}A^{n-k,k}(M).
\end{align}

By Hodge theory, we have
\begin{eqnarray}\label{20020}
\begin{split}
[\Omega^c_p(\tau)]&=[\Omega_p]+\sum_{i=1}^N
\tau_i[\phi_i\lrcorner\Omega_p]+\left[\mathbb{H}(\partial_{M_p}(
\sum_{|I|\geq
2}\tau^I\psi_I) )\right]+[\mathbb{H}(a(\tau))]\\
&=[\Omega_p]+\sum_{i=1}^N
\tau_i[\phi_i\lrcorner\Omega_p]+\left[\mathbb{H}(a(\tau))\right].
\end{split}
\end{eqnarray}
Let $A(\tau)=[\mathbb{H}(a(\tau))]$, then (\ref{a(tau)}) shows that
$A(\tau)\in \bigoplus_{k=2}^n H^{n-k,k}(M)$ and
$A(\tau)=O(|\tau|^2)$ which denotes the terms of order at least $2$ in $\tau$.
\end{proof}

In fact we have the following expansion of the canonical
family of $(n,0)$-classes up to order $2$ in $\tau$,
\begin{align*}
[\Omega^c_p(\tau)]=[\Omega_p]+\sum_{i=1}^N
\tau_i[\phi_i\lrcorner\Omega_p]+\frac{1}{2}\sum_{i,j}\tau_i\tau_j
\left [ \h(\phi_i\wedge\phi_j\lrcorner\Omega_p) \right ]+\Xi(\tau),
\end{align*}
where $\Xi(\tau)=O(|\tau|^3)$ denotes terms of order at least $3$ in $\tau$, and
$\Xi(\tau)\in \bigoplus_{k=2}^n H^{n-k,k}(M)$. This will not be
needed in this paper.

\section{Affine structure on the Teichm\"uller space}\label{AffineonT}
In Section \ref{boundedness of Phi}, we fix a base point $p\in\mathcal{T}$ and introduce the unipotent space $N_+\subseteq \check{D}$, which is biholomorphic to $\mathbb{C}^d$. Then we explain that the image $\Phi(\mathcal{T})$ is bounded in $N_+\cap D$ with respect to the Euclidean metric on $N_+$. In Section \ref{affineonT}, using the property that $\Phi(\mathcal{T})\subseteq N_+$ where $A\subseteq N_+$ is an Abelian subgroup, we define a holomorphic map $\Psi:\,\mathcal{T}\rightarrow A\cap D \subset A\simeq\mathbb{C}^N$. Then we use local Torelli theorem to show that $\Psi$ defines a local coordinate chart around each point in $\mathcal{T}$, and this shows that $\Psi: \,\mathcal{T}\rightarrow A\cap D$ defines a holomorphic affine structure on $\mathcal{T}$. In this section we will use some properties of the period domain from Lie group and Lie algebra point of view, which can be found in Appendix \ref{preliminary}. The detailed proofs of the main results in this section are all contained in \cite{LS}.

\subsection{Boundedness of the period map}\label{boundedness of Phi}          
Now let us fix the base point $p\in \mathcal{T}$ with $\Phi(p)\in D$. Then according to Remark \ref{n+} in the appendix, $N_+$ can be viewed as a subset in $\check{D}$ by identifying it with its orbit in $\check{D}$ with base point $\Phi(p)$.               
  Let us also fix an adapted basis $(\eta_0, \cdots, \eta_{m-1})$ for the Hodge decomposition of the base point $\Phi(p)\in D$. Then we can identify elements in $N_+$ with nonsingular block lower triangular matrices whose diagonal blocks are all identity submatrix. We define 
\begin{align*}
\check{\mathcal{T}}=\Phi^{-1}(N_+).
\end{align*} 
At the base point $\Phi(p)=o\in N_+\cap D$, the tangent space $T_{o}^{1,0}N_+=T_o^{1,0}D\simeq \mathfrak{n}_+\simeq N_+$, then the Hodge metric on $T_o^{1,0}D$ induces an Euclidean metric on $N_+$. In the proof of the following  theorem, for simplicity we require all the root vectors to be unit vectors with respect to this Euclidean metric.

\begin{theorem}\label{locallybounded}
The restriction of the period map $\Phi:\,\check{\mathcal{T}}\rightarrow N_+$ is bounded in $N_+$ with respect to the Euclidean metric on $N_+$. 
\end{theorem}


The proof of this theorem depends on a slight extension of Harish-Chandra's proof of his famous embedding theorem of the Hermitian symmetric domains as bounded domains in complex Euclidean spaces. One may refer to Lemma 7 and Lemma 8 at pp.~582--583 in \cite{HC}, Proposition 7.4 at pp.~385 and Ch VIII $\S7$ at pp.~382--396 in \cite{Hel}, Proposition 1 at pp.~91 and Proof of Theorem 1 at pp.~95--97 in \cite{Mok}, and Lemma 2.2.12 at pp.~141-142 and $\S5.4$ in \cite{Xu} for more details. See \cite{LS} for the detail of the proof of the above theorem.


According to the definition of $\check{\mathcal{T}}$, it is not hard to conclude the following lemma, one may refer to \cite{LS} for the complete proof. \begin{lemma}\label{codimension}The subset $\check{\mathcal{T}}$ is an open dense submanifold in $\mathcal{T}$, and $\mathcal{T}\backslash \check{\mathcal{T}}$ is an analytic subvariety of $\mathcal{T}$ with $\text{codim}_{\mathbb{C}}(\mathcal{T}\backslash \check{\mathcal{T}})\geq 1$. 
\end{lemma}

\begin{corollary}\label{image}The image of $\Phi:\,\mathcal{T}\rightarrow D$ lies in $N_+\cap D$ and is bounded with respect to the Euclidean metric on $N_+$. 
\end{corollary}
\begin{proof}According to Lemma \ref{codimension}, $\mathcal{T}\backslash\check{\mathcal{T}}$ is an analytic subvariety of $\mathcal{T}$ and the complex codimension of $\mathcal{T}\backslash\check{\mathcal{T}}$ is at least one; by Theorem\ref{locallybounded}, the holomorphic map $\Phi:\,\check{\mathcal{T}}\rightarrow N_+\cap D$ is bounded in $N_+$ with respect to the Euclidean metric. Thus by the Riemann extension theorem, there exists a holomorphic map $\Phi': \,\mathcal{T}\rightarrow N_+\cap D$ such that $\Phi'|_{_{\check{\mathcal{T}}}}=\Phi|_{_{\check{\mathcal{T}}}}$. Since as holomorphic maps, $\Phi'$ and $\Phi$ agree on the open subset $\check{\mathcal{T}}$, they must be the same on the entire $\mathcal{T}$. Therefore, the image of $\Phi$ is in $N_+\cap D$, and the image is bounded with respect to the Euclidean metric on $N_+.$ As a consequence, we also get $\mathcal{T}=\check{\mathcal{T}}=\Phi^{-1}(N_+).$ \end{proof}
\subsection{Holomorphic affine structure on the Teichm\"uller space}\label{affineonT}
We first review the definition of complex affine manifolds.
One may refer to page~215 of
\cite{Mats} for more details.
\begin{definition}
Let $M$ be a complex manifold of complex dimension $n$. If there
is a coordinate cover $\{(U_i,\,\phi_i);\, i\in I\}$ of M such
that $\phi_{ik}=\phi_i\circ\phi_k^{-1}$ is a holomorphic affine
transformation on $\mathbb{C}^n$ whenever $U_i\cap U_k$ is not
empty, then $\{(U_i,\,\phi_i);\, i\in I\}$ is called a complex
affine coordinate cover on $M$ and it defines a holomorphic affine structure on $M$.
\end{definition}

There is a canonical Euclidean subspace $A\subset N_{+}$ given by $\mathfrak{a}=\Phi_*(\text{T}_p^{1,0}\T)\subset \mathfrak{n}_{+}$, the abelian subalgebra, and $A=\exp(\mathfrak{a})$ the corresponding Lie group. We then define the holomorphic map 
$$\Psi :\, \check{\T} \to A\cap D,$$
by $\Psi=P \circ \Phi|_{\check{\T}}$, where $P$ is the projection map from $N_{+}\cap D$ to $A\cap D$. By Theorem \ref{image}, we can extend the holomorphic map $\Psi :\, \check{\T} \to A\cap D$ over $\T$ as
$$\Psi :\, {\T} \to A\cap D,$$
such that $\Psi=P \circ \Phi$.

By the local Torelli theorem for Calabi--Yau manifolds which implies that $\Psi$ is nondegenerate, and the definition of holomorphic affine structure, we can prove with the main theorem of this section. Again we refer the reader to \cite{LS} for the detail of its proof.
\begin{theorem}The holomorphic map $\Psi :\, {\T} \to A\cap D$ defines a local coordinate around each point $q\in{\mathcal{T}}$. Thus the map $\Psi$ itself gives a global holomorphic coordinate for ${\mathcal{T}}$ with the transition maps all identity maps. In particular, the global holomorphic coordinate $\Psi :\, {\T} \to A\cap D \subset A\simeq \C^{N}$ defines a holomorphic affine structure on ${\mathcal{T}}$. Therefore, ${\mathcal{T}}$ is a complex affine manifold.
\end{theorem}

Note that this affine structure on ${\mathcal{T}}$ depends on the choice of the base point $p$. Affine structures on ${\mathcal{T}}$ defined in this ways by fixing different base point may not be compatible with each other. 

\section{Hodge metric completion of the Torelli space}\label{THm}

This section contains a review of our extension of the affine structure to the Hodge metric completion $\mathcal{T}^H$ of the Torelli space $\mathcal{T}'$, as well as its consequences including a proof of the global Torelli theorem for polarized and marked Calabi--Yau manifolds on the Torelli space.

Recall that in Section \ref{section period map}, we have introduced the Hodge metric on the period domain $D$ and its pull backs on $\Z$, $\T$ and $\T'$ via the period maps from the corresponding spaces. 
By the work of Viehweg in \cite{v1}, we know that $\mathcal{Z}_m$ is quasi-projective and consequently we can find a smooth projective compactification $\bar{\mathcal{Z}}_{m}$ such that $\mathcal{Z}_m$ is Zariski open in $\bar{\mathcal{Z}}_{m}$ and the complement $\bar{\mathcal{Z}}_{m}\backslash\mathcal{Z}_m$ is a divisor of normal crossings. Therefore, $\mathcal{Z}_m$ is dense and open in $\bar{\mathcal{Z}}_{m}$ with the complex codimension of the complement $\bar{\mathcal{Z}}_{m}\backslash \mathcal{Z}_m$ at least one. Moreover as $\bar{\mathcal{Z}}_{m}$ is a compact space, it is a complete space.

Let us now take $\mathcal{Z}^H_{m}$ to be the completion of $\mathcal{Z}_m$ with respect to the Hodge metric. Then $\mathcal{Z}_{m}^H$ is the smallest complete space with respect to the Hodge metric that contains $\mathcal{Z}_m$. Although the compact space $\bar{\mathcal{Z}}_{m}$ may not be unique, the Hodge metric completion space $\mathcal{Z}^H_{m}$ is unique up to isometry, and is identified to the Griffiths extension as given in Theorem 9.5 and Theorem 9.6 in \cite{Griffiths3}. In particular, $\mathcal{Z}^H_{m}\subseteq\bar{\mathcal{Z}}_{m}$ and thus the complex codimension of the complement $\mathcal{Z}^H_{m}\backslash \mathcal{Z}_m$ is at least one. By Lemma 2.7 in \cite{LS}, we conclude the following useful result.
\begin{itemize}
\item The mectic completion $\mathcal{Z}^H_{m}$ is a dense and open smooth submanifold in $\bar{\mathcal{Z}}_m$ with $\text{codim}_{\mathbb{C}}(\mathcal{Z}^H_{m}\backslash \mathcal{Z}_m)\geq 1$, and $\mathcal{Z}^H_{m}\backslash \mathcal{Z}_m$ consists of those points in $\bar{\mathcal{Z}}_{m}$ around which the so-called Picard-Lefschetz transformations are trivial.
Moreover the extended period map $\Phi^H_{_{\mathcal{Z}_m}}:\, \mathcal{Z}^H_{m}\to D/\Gamma$ is proper and holomorphic.
\end{itemize}
The proof of this conclusion uses the Griffiths extension of the period map as wells as the basic definitions of metric completion. 

We denote by $\mathcal{T}^H_{m}$ the universal cover of $\mathcal{Z}_{m}^{H}$, which is now a connected and simply connected complete smooth complex manifold.
By using elementary topology argument, we show that the following commutative diagram holds:
\begin{align*}
\xymatrix{\mathcal{T}\ar[r]^{i_{m}}\ar[d]^{\pi_m}&\mathcal{T}^H_{m}\ar[d]^{\pi_{m}^H}\ar[r]^{{\Phi}^{H}_{m}}&D\ar[d]^{\pi_D}\\
\mathcal{Z}_m\ar[r]^{i}&\mathcal{Z}^H_{m}\ar[r]^{{\Phi}_{_{\mathcal{Z}_m}}^H}&D/\Gamma,
}
\end{align*}
where $i$ is the inclusion map, $i_m$ is a lifting of $i\circ \pi_{m}$, and $\Phi^H_{m}$ is a lifting of $\Phi^H_{_{\mathcal{Z}_m}}\circ \pi_{m}^H$, and we fix a suitable choice of $i_m$ and $\Phi^H_{m}$ such that $\Phi=\Phi^H_{m}\circ i_m$. 

Let us denote $\mathcal{T}_m:=i_{m}(\mathcal{T})$ and the restriction map $\Phi_m=\Phi^H_{m}|_{\mathcal{T}_m}$, then we also have $\Phi=\Phi_m\circ i_m$. Moreover, it is easy to see that $\Phi_m$ is also bounded by Corollary \ref{image}. 
By using argument of basic algebraic topology, we prove that the image $\mathcal{T}_m$ equals to the preimage $(\pi^H_{m})^{-1}(\mathcal{Z}_{m})$, and moreover $i_m:\, \T \to \T_m$ is a covering map. See Proposition 2.8 of \cite{LS} for the details.
Therefore, $\mathcal{T}_m$ is a connected open complex submanifold in $\mathcal{T}^H_{m}$ and $\text{codim}_{\mathbb{C}}(\mathcal{T}^H_{m}\backslash\mathcal{T}_m)\geq 1$. It is easy to see that $\mathcal{Z}_{m}^{H}\setminus \mathcal{Z}_{m}$ is an analytic subvariety of $\mathcal{Z}_{m}^{H}$, and hence the set $\mathcal{T}^H_{m}\backslash\mathcal{T}_m$ is also an analytic subvariety of $\mathcal{T}^H_{m}$.

Recall that we have fixed a base point $p\in\mathcal{T}$ and an adapted basis $\{\eta_0, \cdots, \eta_{m-1}\}$ for the Hodge decomposition of the base point $\Phi(p)\in D$. With the fixed base point in $D$, we can identify $N_+$ with its unipotent orbit in $\check{D}$.  Then applying the Riemann extension theorem to the bounded map $\Phi_m: \mathcal{T}_m\rightarrow N_+\cap D$, we obtain the following lemma.
\begin{lemma}\label{Riemannextension}
The map $\Phi^{H}_{m}$ is a bounded holomorphic map from $\mathcal{T}^H_{m}$ to $N_+\cap D$.
\end{lemma}

Now by using the Riemann extension theorem, we can extend the map $\Psi_m :\, \T_m \to A\cap D$, which is also bounded by Lemma \ref{Riemannextension}, to its extension
$$\Psi^{H}_{m} :\, \T^{H}_{m} \to A\cap D,$$
such that $\Psi^{H}_{m}=P\circ \Phi^{H}_{m}$, where $P$ is the projection map from $N_{+}\cap D$ to its subspace $A\cap D$.
Moreover, we also have $\Psi=P\circ\Phi=P\circ \Phi^H_{m}\circ i_m=\Psi^H_{m}\circ i_m$. Then by identifying the tangent bundle of $\T_m$ with Hodge bundles and the property that the holomorphic map $\Psi$ defines the holomorphic affine structure on $\mathcal{T}$, we can prove the following theorem.

\begin{theorem}\label{THmaffine}The holomorphic map $\Psi^H_{m}: \,\mathcal{T}^H_{m}\rightarrow A\cap D$ is a local embedding, therefore it defines a global holomorphic affine structure on $\mathcal{T}^H_{m}$. \end{theorem}

Now by the completeness of $\mathcal{T}^H_{m}$ with Hodge metric,  $\Psi^H_{m}$ is an isometry with Hodge metric at every point of $\mathcal{T}^H_{m}$, and a result of Griffiths and Wolf \cite{GW}, we can conclude that the holomorphic map $\Psi^H_{m}: \,\mathcal{T}^H_{m}\rightarrow A\cap D$ is a covering map. 
It is easy to show that $A\cap D$ is simply connected, therefore  $\Psi^H_{m}$ must be a biholomorphic map.


Moreover, as $\Psi^H_{m}=P\circ \Phi^{H}_{m}$, where $P$ is the projection map and $\Phi^{H}_{m}$ is a bounded map, we may conclude the injectivity of $\Phi^{H}_{m}$. To conclude, we have the following theorem, and one may refer to \cite{LS} for its detailed proof. 
\begin{theorem}\label{injectivityofPhiH}For any $m\geq 3$, the holomorphic map $\Psi^H_{m}:\, \mathcal{T}^H_{m}\to A\cap D$ is an injection and hence bi-holomorphic. In particular, the completion space $\mathcal{T}^H_{m}$ is a bounded domain $A\cap D$ in $A\simeq \mathbb{C}^N$. Moreover, the holomorphic map $\Phi^H_{m}: \,\mathcal{T}^H_{m}\rightarrow N_+\cap D$ is also an injection.
\end{theorem}
\section{Proof of global Torelli on Torelli space}\label{Proof}

In this section we describe the main idea of the proof of the global Torelli theorem on Torelli space. That is to prove that the period map $\Phi':\, \T'\to D$ is an injective holomorphic map. The key step of the proof is to first show that  $\T_m=i_m(\T) $ is independent of $m$, and that $\T_m$ is bi-holomorphic to the Torelli space $\T'$. 

To start, one first notes that one direct consequence of Theorem \ref{injectivityofPhiH} is the following proposition. 
\begin{proposition}\label{indepofm} For any $m, m'\geq 3$,
the complete complex manifolds $\mathcal{T}^H_m$ and $\mathcal{T}^H_{m'}$ are biholomorphic to each other.
\end{proposition}


This allows us to introduce the new simplified notations by dropping the level $m$.

\begin{definition}We define the complete complex manifold $\mathcal{T}^H=\mathcal{T}^H_{m}$, the holomorphic map $i_{\mathcal{T}}: \,\mathcal{T}\to \mathcal{T}^H$ by $i_{\mathcal{T}}=i_m$, and the extended period map $\Phi^H:\, \mathcal{T}^H\rightarrow D$ by $\Phi^H=\Phi^H_{m}$ for any $m\geq 3$.
In particular, with these new notations, we have the commutative diagram
\begin{align}\label{main diagram}
\xymatrix{\mathcal{T}\ar[r]^{i_{\mathcal{T}}}\ar[d]^{\pi_m}&\mathcal{T}^H\ar[d]^{\pi^H_{m}}\ar[r]^{{\Phi}^{H}}&D\ar[d]^{\pi_D}\\
\mathcal{Z}_m\ar[r]^{i}&\mathcal{Z}^H_{m}\ar[r]^{{\Phi}_{_{\mathcal{Z}_m}}^H}&D/\Gamma.
}
\end{align}
\end{definition}

Let $\T_{0}\subset \T^{H}$ be defined by $\T_{0}:=i_{\T}(\T)$. Since $\mathcal{T}_{0}\simeq \mathcal{T}_m=(\pi^H_{m})^{-1}(\mathcal{Z}_m)$ for any $m\ge 3$, $\pi_m^H:\, \mathcal{T}_0\to \mathcal{Z}_m$ is a covering map. Thus the fundamental group of $\mathcal{T}_0$ is a subgroup of the fundamental group of $\mathcal{Z}_m$, that is, $ \pi_1(\mathcal{T}_0)\subseteq \pi_1(\mathcal{Z}_m)$ for any $m\geq 3$.
Moreover, the universal property of the universal covering map $\pi_m: \,\mathcal{T}\to \mathcal{Z}_m$ with the identity $i_m\circ \pi_m=\pi_{m}^H|_{\mathcal{T}_0}\circ i_{\mathcal{T}}$
implies that $i_{\mathcal{T}}:\,\mathcal{T}\to \mathcal{T}_0$ is also a covering map, as discussed in Section \ref{THm}.

On the other hand, let $\{m_k\}_{k=1}^{\infty}$ be a sequence of positive integers such that $m_k< m_{k+1}$ and $m_k|m_{k+1}$ for each $k\geq 1$. From the discussion of Lecture 10 of \cite{Popp}, or Page 5 of \cite{sz}, there is a natural covering map from $\mathcal{Z}_{m_{k+1}}$ to  $\mathcal{Z}_{m_{k}}$ for each $k$.
Thus the fundamental group $\pi_1(\mathcal{Z}_{m_{k+1}})$ is a subgroup of $\pi_1(\mathcal{Z}_{m_{k}})$ for each $k$. Hence, the inverse system of fundamental groups
\begin{align*}
\xymatrix{
\pi_1(\mathcal{Z}_{m_1})&\pi_1(\mathcal{Z}_{m_2})\ar[l]&\cdots\cdots\ar[l]&\pi_1(\mathcal{Z}_{m_k})\ar[l]&\cdots\ar[l]
}
\end{align*}
has an inverse limit, which is the fundamental group of the Torelli space $\mathcal{T}'$. Since $\pi_1(\mathcal{T}_0)\subseteq \pi_1(\mathcal{Z}_{m_{k}})$ for any $k$, we have the inclusion $\pi_1(\mathcal{T}_0)\subseteq\pi_1(\mathcal{T}')$. This implies that $\T_0$ is a covering of $\T'$. Let $\pi_{0}:\, \T_{0}\to \T'$ be the covering map. Then the covering map $\pi_{0}:\, \T_{0}\to \T'$ together with diagram \eqref{cover1} and \eqref{main diagram} implies that the following diagram is commutative
$$\xymatrix{
\T \ar[dr]^-{\pi}\ar[rr]^-{i_{\T}} \ar[dd]^-{\pi_m} &&\T_{0}\ar[dl]^-{\pi_{0}}  \ar[dd]^-{\pi_{m}^{H}|_{\T_{0}}}\ar[rr]^-{\Phi^{H}|_{\T_{0}}}  &&D\ar[dd]^-{\pi_{D}}\\
&\T'\ar[dl]_-{\pi'_{m}}\ar[urrr]^-{\Phi'}&&&\\
\mathcal{Z}_{m} \ar[rr]^-{i_{m}} &&\mathcal{Z}_{m}^{H} \ar[rr]^-{\Phi_{\mathcal{Z}_{m}^{H}}} &&D/\Gamma .}$$

Since $\Phi^{H}:\, \T^{H}\to D$ is injective, so is the restriction map $\Phi^{H}|_{\T_{0}}:\, \T_{0}\to D$, which implies the injectivity of the map $\pi_{0}:\, \T_{0}\to \T'$ by the relation $\Phi^{H}|_{\T_{0}}=\Phi'\circ \pi_0$.
Therefore $\pi_{0}:\, \T_{0}\to \T'$ is biholomorphic. 

In a summary, we have proved the following proposition. Note that in \cite{LS} we give a more elementary proof of this result by directly constructing the map $\pi_0$.
\begin{proposition}\label{i_T}
Let $\T_{0}\subset \T^{H}$ be defined by $\T_{0}:=i_{\T}(\T)$. Then $\T_{0}$ is biholomorphic to the Torelli space $\T'$.
\end{proposition}

From Proposition \ref{i_T}, we can see that the complete complex manifold $\mathcal{T}^H$ is actually the completion space of the Torelli space $\T'$ with respect to the Hodge metric.

Since the restriction map $\Phi^{H}|_{\T_{0}}$ is injective and $\Phi'=\Phi^{H}|_{\T_{0}}\circ (\pi_{0})^{-1}$, we get the global Torelli theorem for the period map $\Phi':\, \T'\to D$ from the Torelli space to the period domain as follows.
\begin{theorem}[Global Torelli theorem]
The period map $\Phi':\, \mathcal{T}'\rightarrow D$ is injective.
\end{theorem}

We refer the reader to \cite{LS} for all of the details in the above proofs. 

\section{Global canonical sections of the holomorphic $(n,0)$-classes}\label{globalSection}
In this section we extend the local canonical sections of the holomorphic $(n,0)$-classes to the global canonical sections on the Hodge completion space $\T^{H}$, by constructing global holomorphic sections of the Hodge bundles $\{F^k\}_{k=0}^n$ over $\mathcal{T}^H$. Same results hold on the Techm\"uller space $\T$ by pulling back through the covering map $i_\T:\, \T\to \T^H$.

Recall that we have fixed a base point $p\in\mathcal{T}^{H}$ and an adapted basis $\{\eta_0, \cdots, \eta_{m-1}\}$ for the Hodge decomposition of the base point $\Phi^{H}(p)\in D$. With the fixed base point in $D$, we can identify the unipotent group $N_+$ with its unipotent orbit in $\check{D}$ by identifying an element $c\in N_+$ with $[c]=cB$ in $\check{D}$.

On one hand, as we have fixed an adapted basis $\{\eta_0, \cdots, \eta_{m-1}\}$ for the Hodge decomposition of the base point, the elements in $G_{\mathbb{C}}$ can be identified with a subset of the nonsingular block matrices. In particular, the elements in $N_+$ can be realized as nonsingular block lower triangular matrices whose diagonal blocks are all identity submatrix. Namely, for any element $\{F^k_o\}_{k=0}^n\in N_+\subseteq\check{D}$, there exists a unique nonsingular block lower triangular matrices $A(o)\in G_{\mathbb{C}}$ such that
$(\eta_0, \cdots, \eta_{m-1})A(o)$ is an adapted basis for the Hodge filtration $\{F^k_o\}\in N_+$ that represents this element in $N_+$. 

Similarly, any elements in $B$ can be realized as nonsingular block upper triangular matrices in $G_{\mathbb{C}}$. Moreover, as $\check{D}=G_{\mathbb{C}}/B$, we have that for any $U\in G_{\mathbb{C}}$, which is a nonsingular block upper triangular matrix, $(\eta_0, \cdots, \eta_{m-1})A(o)U$ is also an adapted basis for the Hodge filtration $\{F^k(o)\}_{k=0}^n$. Conversely, if $(\zeta_0, \cdots, \zeta_{m-1})$ is an adapted basis for the Hodge filtration $\{F^k_o\}_{k=0}^n$, then there exists a unique $U\in G_{\mathbb{C}}$ such that $(\zeta_0, \cdots, \zeta_{m-1})=(\eta_0, \cdots, \eta_{m-1})A(o)U$. For any $q\in\mathcal{T}^{H}$, let us denote the Hodge filtration at $q\in \mathcal{T}^{H}$ by $\{F_q^k\}_{k=0}^n$, and we have that $\{F_q^k\}_{k=0}^n\in N_+\cap D$ by Theorem \ref{injectivityofPhiH}. Thus there exists a unique nonsingular block lower triangular matrices $\tilde{A}(q)$ such that $(\eta_0, \cdots, \eta_{m-1})\tilde{A}(q)$ is an adapted basis for the Hodge filtration $\{F^k_q\}_{k=0}^{n}$.

On the other hand, for any adapted basis $\{\zeta_0(q), \cdots, \zeta_{m-1}(q)\}$ for the Hodge filtration $\{F^k_q\}_{k=0}^n$ at $q$, we know that there exists an $m\times m$ transition matrix $A(q)$ such that
$(\zeta_0(q), \cdots, \zeta_{m-1}(q))=(\eta_0, \cdots, \eta_{m-1})A(q).$
Moreover, we set the blocks of $A(q)$ as in \eqref{block} and denote the ${(i,j)}$-th block of $A(q)$ by $A^{i,j}(q)$.

As both $(\eta_0, \cdots, \eta_{m-1})\tilde{A}(q)$ and $(\eta_0, \cdots, \eta_{m-1})A(q)$ are adapted bases for the Hodge filtration for $\{F^k_q\}_{k=0}^n$, there exists a $U\in G_{\mathbb{C}}$ which is a block nonsingular upper triangular matrix such that $(\eta_0, \cdots, \eta_{m-1})\tilde{A}(q)U=(\eta_0, \cdots, \eta_{m-1})A(q).$
Therefore, we conclude that
\begin{align}\label{upper lower}\tilde{A}(q)U=A(q).
\end{align}
where $\tilde{A}(q)$ is a nonsingular block lower triangular matrix in $G_{\mathbb{C}}$ with all the diagonal blocks equal to identity submatrix, while $U$ is a block upper triangular matrix in $G_{\mathbb{C}}$. However, according to basic linear algebra, we know that a nonsingular matrix $A(q)\in G_{\mathbb{C}}$ have the decomposition of the type in \eqref{upper lower} if and only if the principal submatrices $[A^{i,j}(q)]_{0\leq i,j\leq n-k}$ are nonsingular for all $0\leq k\leq n$.

To conclude, by Theorem \ref{injectivityofPhiH}, we have that $\Phi^{H}(q)\in N_+$ for any $q\in \mathcal{T}^{H}$. Therefore, for any adapted basis $(\zeta_0(q), \cdots, \zeta_{m-1}(q))$, there exists a nonsingular block matrix $A(q)\in G_{\mathbb{C}}$ with $\det[A^{i,j}(q)]_{0\leq i,j\leq n-k}\neq 0$ for any $0\leq k\leq n$ such that
$(\zeta_0(q), \cdots, \zeta_{m-1}(q))=(\eta_0, \cdots, \eta_{m-1})A(q).$
Let $\{F^k_p\}_{k=0}^n$ be the reference Hodge filtration at the base point $p\in\mathcal{T}^{H}$.
For any point $q \in \mathcal{T}^H$ with the corresponding Hodge filtrations $\{F^k_q\}_{k=0}^n$, we define the following maps
\begin{align*}
P^k_q: \, F^k_q\to F^k_p \quad\text{for any}\quad 0\leq k\leq n
\end{align*}
to be the projection map with respect to the Hodge decomposition at the base point $p$. With the above notation, we therefore have the following lemma.  
\begin{lemma}\label{globalTransversality}
For any point $q\in\mathcal{T}^H$ and $0\leq k\leq n$, the map $P_q^k:\, F^k_q\to F^k_p$ is an isomorphism. Furthermore, $P_q^k$ depends on $q$ holomorphically.
\end{lemma}
\begin{proof}
We have already fixed $\{\eta_0, \cdots, \eta_{m-1}\}$ as an adapted basis for the Hodge decomposition of the Hodge structure at the base point $p$. Thus it is also the adapted basis for the Hodge filtration $\{F^k_p\}_{k=0}^n$ at the base point.
For any point $q\in \mathcal{T}^{H}$, let $\{\zeta_0, \cdots, \zeta_{m-1}\}$ be an adapted basis for the Hodge filtration $\{F^k_q\}_{k=0}^n$ at $q$. Let $[A^{i,j}(q)]_{0\leq i,j\leq n}\in G_{\mathbb{C}}$ be the transition matrix between the basis $\{\eta_0,\cdots, \eta_{m-1}\}$ and $\{\zeta_0, \cdots, \zeta_{m-1}\}$ for the same vector space $H^{n}(M, \mathbb{C})$. We have showed that $[A^{i,j}(q)]_{0\leq i,j\leq n-k}$ is nonsingular for all $0\leq k\leq n$.

On the other hand, the submatrix $[A^{i,j}(q)]_{0\leq j\leq n-k}$ is the transition matrix between the bases of $F^k_q$ and $F^0_p$ for any $0\leq k\leq n$, that is
\begin{align*}
(\zeta_0(q), \cdots, \zeta_{f^k-1}(q))=(\eta_0, \cdots, \eta_{m-1})[A^{i,j}(q)]_{0\leq j\leq n-k} \quad\text{for any}\quad 0\leq k\leq n,
\end{align*}
where $(\zeta_0(q), \cdots, \zeta_{f^k-1}(q))$ and $(\eta_0, \cdots, \eta_{m-1})$ are the bases for $F^k_q$ and $F^0_p$ respectively.
Thus the matrix of $P_q^k$ with respect to $\{\eta_0, \cdots, \eta_{f^k-1}\}$ and $\{\zeta_0, \cdots, \zeta_{f^k-1}\}$ is the first $(n-k+1)\times (n-k+1)$ principal submatrix $[A^{i,j}(q)]_{0\leq i,j\leq {n-k}}$ of $[A^{i,j}(q)]_{0\leq i,j\leq n}$. Now since $ [A^{i,j}(q)]_{0\leq i,j\leq {n-k}}$ for any $0\leq k\leq n$ is nonsingular, we conclude that the map $P_q^k$ is an isomorphism for any $0\leq k\leq n$.

From our construction, it is clear that the projection $P_q^k$ depends on $q$ holomorphically.

\end{proof}

Using this lemma, we are ready to construct the global holomorphic sections of Hodge bundles over $\mathcal{T}^H$. 
For any $0\leq k\leq n$, we know that $\{\eta_0, \eta_1, \cdots, \eta_{f^{k}-1}\}$ is an adapted basis of the Hodge decomposition of $F^k_p$ for any $0\leq k\leq n$. Then we define the sections
\begin{align}\label{sections}
s_i: \mathcal{T}^H\rightarrow F^{k}, \quad q\mapsto (P_q^k)^{-1}(\eta_i)\in F^k_q\quad\text{for any}\quad 0\leq i\leq f^{k}-1.
\end{align}
%
Lemma \ref{globalTransversality} implies that $\{(P_q^k)^{-1}(\eta_i)\}_{i=0}^{f^k-1}$ form a basis of $F^k_q$ for any $q\in \mathcal{T}^H$. In fact, we have proved the following theorem for polarized and marked Calabi--Yau manifolds.
\begin{theorem}\label{trivial bundle}
For all $0\leq k\leq n$, the Hodge bundles $F^k$ over $\mathcal{T}^H$ are trivial bundles, and the trivialization can be obtained by $\{s_i\}_{0\leq i\leq f^k-1}$ which is defined in \eqref{sections}. In particular, the section $s_0: \mathcal{T}^H\to F^n$ is a global nowhere zero section of the Hodge bundle $F^n$ for Calabi--Yau manifolds.
\end{theorem}

With the adapted basis at the base point $p\in \T^{H}$, we can also see $\Phi_*(\text{T}^{1,0}_{p}(\T))=\mathfrak{a}\subset \mathfrak{n}_{+}$ as a block lower triangle matrix whose diagonal elements are zero. Moreover by local Torelli theorem for Calabi--Yau manifolds, we can conclude that $\mathfrak{a}$ is isomorphic to its $(1,0)$-block as vector spaces, see \eqref{block} for the definition. Let $(\tau_{1},\cdots ,\tau_{N})^{T}$ be the $(1,0)$-block of $\mathfrak{a}$. 
Sine the affine structure on $A$ is induced by $\exp:\, \mathfrak{a}\to A$ which is an isomorphism, $(\tau_{1},\cdots ,\tau_{N})^{T}$ also defines a global affine structure on $A$, and hence on $\T^{H}$. We denote it by
$$\tau^{H}:\, \T^{H} \to \C^{N},\quad  q\mapsto (\tau_{1}(q),\cdots ,\tau_{N}(q)).$$
Note that from linear algebra, it is easy to see that the $(1,0)$-block of $A=\exp(\mathfrak{a})$ is still $(\tau_{1},\cdots ,\tau_{N})^{T}$. Hence the affine map defined as above can be constructed as the $(1,0)$-block of the image of the period map.
More precisely, let $P^{1,0}:\, N_{+}\to \C^{N}$ to be the projection of the matrices in $N_{+}$ onto their $(1,0)$-blocks. Then the affine map is
$$\tau^{H}=P^{1,0}\circ \Phi^{H}: \, \T^{H} \to \C^{N}.$$
Moreover $\tau^{H}$ is injective, and hence it defines another embedding of $\T^{H}$ into $\C^{N}$.
\begin{remark}
The global affine coordinate $\tau^{H}: \, \T^{H} \to \C^{N}$ extends the local holomorphic flat coordinate around the base point $p\in \T^{H}$ in Theorem \ref{flatcoord}.
\end{remark}

Using the same notation as in Lemma \ref{expcoh}, we are ready to prove the following theorem for Calabi--Yau manifolds,
\begin{theorem}\label{globalExp}
Choose $[\Omega_p]=\eta_0$, then the section $s_0$ of $F^n$ is a global holomorphic extension of the local section $[\Omega^c_p(\tau)]$.
\end{theorem}
\begin{proof}
Because both $s_0$ and $[\Omega^c_p(\tau)]$ are holomorphic sections of $F^n$, we only need to show that $s_0|_{U_p}=[\Omega^c_p(\tau)]$. In fact, from the expansion formula \eqref{cohexp10}, we have that for any $q\in U_p$
\begin{align*}
P^n_q([\Omega^c_p(\tau(q))])=[\Omega_p]=\eta_0.
\end{align*}
Therefore, $[\Omega^c_p(\tau(q))]=(P^n_q)^{-1}(\eta_0)=s_0(q)$ for any point $q\in U_p$.
\end{proof}

\appendix
\section{Classifying spaces and Hodge structures from Lie group point of view. }\label{preliminary}
In this appendix, we review some properties of the period domain from Lie group and Lie algebra point of view. All of the results in this section is well-known to the experts in the subject. The purpose to give details is to fix notations. One may either skip this section or refer to \cite{GS} and \cite{schmid1} for most of the details. 

The orthogonal group of the bilinear form $Q$ in the definition of Hodge structure is a linear algebraic group, defined over $\mathbb{Q}$. Let us simply denote $H_{\mathbb{C}}=H^n(M, \mathbb{C})$ and $H_{\mathbb{R}}=H^n(M, \mathbb{R})$. The group of the $\mathbb{C}$-rational points is
\begin{align*}
G_{\mathbb{C}}=\{ g\in GL(H_{\mathbb{C}})|~ Q(gu, gv)=Q(u, v) \text{ for all } u, v\in H_{\mathbb{C}}\},
\end{align*}
which acts on $\check{D}$ transitively. The group of real points in $G_{\mathbb{C}}$ is
\begin{align*}
G_{\mathbb{R}}=\{ g\in GL(H_{\mathbb{R}})|~ Q(gu, gv)=Q(u, v) \text{ for all } u, v\in H_{\mathbb{R}}\},
\end{align*}
which acts transitively on $D$ as well.

Consider the period map $\Phi:\, \mathcal{T}\rightarrow D$. Fix a point $p\in \mathcal{T}$ with the image $o:=\Phi(p)=\{F^n_p\subset \cdots\subset F^{0}_p\}\in D$. The points $p\in \mathcal{T}$ and $o\in D$ may be referred as the base points or the reference points. A linear transformation $g\in G_{\mathbb{C}}$ preserves the base point if and only if $gF^k_p=F^k_p$ for each $k$. Thus it gives the identification
\begin{align*}
\check{D}\simeq G_\mathbb{C}/B\quad\text{with}\quad B=\{ g\in G_\mathbb{C}|~ gF^k_p=F^k_p, \text{ for any } k\}.
\end{align*}
Similarly, one obtains an analogous identification
\begin{align*}
D\simeq G_\mathbb{R}/V\hookrightarrow \check{D}\quad\text{with}\quad V=G_\mathbb{R}\cap B,
\end{align*}
where the embedding corresponds to the inclusion $
G_\mathbb{R}/V=G_{\mathbb{R}}/G_\mathbb{R}\cap B\subseteq G_\mathbb{C}/B.$
The Lie algebra $\mathfrak{g}$ of the complex Lie group $G_{\mathbb{C}}$ can be described as
\begin{align*}
\mathfrak{g}&=\{X\in \text{End}(H_\mathbb{C})|~ Q(Xu, v)+Q(u, Xv)=0, \text{ for all } u, v\in H_\mathbb{C}\}.
\end{align*}
It is a simple complex Lie algebra, which contains
$\mathfrak{g}_0=\{X\in \mathfrak{g}|~ XH_{\mathbb{R}}\subseteq H_\mathbb{R}\}$
as a real form, i.e. $\mathfrak{g}=\mathfrak{g}_0\oplus i \mathfrak{g}_0.$ With the inclusion $G_{\mathbb{R}}\subseteq G_{\mathbb{C}}$, $\mathfrak{g}_0$ becomes Lie algebra of $G_{\mathbb{R}}$. One observes that the reference Hodge structure $\{H^{k, n-k}_p\}_{k=0}^n$ of $H^n(M,{\mathbb{C}})$ induces a Hodge structure of weight zero on
$\text{End}(H^n(M, {\mathbb{C}})),$ namely,
\begin{align*}
\mathfrak{g}=\bigoplus_{k\in \mathbb{Z}} \mathfrak{g}^{k, -k}\quad\text{with}\quad\mathfrak{g}^{k, -k}=\{X\in \mathfrak{g}|XH_p^{r, n-r}\subseteq H_p^{r+k, n-r-k}\}.
\end{align*}
Since the Lie algebra $\mathfrak{b}$ of $B$ consists of those $X\in \mathfrak{g}$ that preserves the reference Hodge filtration $\{F_p^n\subset\cdots\subset F^0_p\}$, one thus has
\begin{align*}
 \mathfrak{b}=\bigoplus_{k\geq 0} \mathfrak{g}^{k, -k}.
\end{align*}
The Lie algebra $\mathfrak{v}_0$ of $V$ is 
$\mathfrak{v}_0=\mathfrak{g}_0\cap \mathfrak{b}=\mathfrak{g}_0\cap \mathfrak{b}\cap\bar{\mathfrak{b}}=\mathfrak{g}_0\cap \mathfrak{g}^{0, 0}.$
With the above isomorphisms, the holomorphic tangent space of $\check{D}$ at the base point is naturally isomorphic to $\mathfrak{g}/\mathfrak{b}$.

Let us consider the nilpotent Lie subalgebra $\mathfrak{n}_+:=\oplus_{k\geq 1}\mathfrak{g}^{-k,k}$. Then one gets the holomorphic isomorphism $\mathfrak{g}/\mathfrak{b}\cong \mathfrak{n}_+$. We take the unipotent group $N_+=\exp(\mathfrak{n}_+)$.

As $\text{Ad}(g)(\mathfrak{g}^{k, -k})$ is in $\bigoplus_{i\geq k}\mathfrak{g}^{i, -i} \text{ for each } g\in B,$ the sub-Lie algebra $\mathfrak{b}\oplus \mathfrak{g}^{-1, 1}/\mathfrak{b}\subseteq \mathfrak{g}/\mathfrak{b}$
defines an Ad$(B)$-invariant subspace. By left translation via $G_{\mathbb{C}}$, $\mathfrak{b}\oplus\mathfrak{g}^{-1,1}/\mathfrak{b}$ gives rise to a $G_{\mathbb{C}}$-invariant holomorphic subbundle of the holomorphic tangent bundle at the base point. It will be denoted by $T^{1,0}_{o, h}\check{D}$, and will be referred to as the holomorphic horizontal tangent bundle at the base point. One can check that this construction does not depend on the choice of the base point. The horizontal tangent subbundle at the base point $o$, restricted to $D$, determines a  subbundle $T_{o, h}^{1, 0}D$ of the holomorphic tangent bundle $T^{1, 0}_oD$ of $D$ at the base point. The $G_{\mathbb{C}}$-invariace of $T^{1, 0}_{o, h}\check{D}$ implies the $G_{\mathbb{R}}$-invariance of $T^{1, 0}_{o, h}D$. 
As another interpretation of this holomorphic horizontal bundle at the base point, one has
\begin{align}\label{horizontal}
T^{1, 0}_{o, h}\check{D}\simeq T^{1, 0}_o\check{D}\cap \bigoplus_{k=1}^n\text{Hom}(F^{k}_p/F^{k+1}_p, F^{k-1}_p/F^{k}_p).
\end{align}
In \cite{schmid1}, Schmid call a holomorphic mapping $\Psi: \,{M}\rightarrow \check{D}$ of a complex manifold $M$ into $\check{D}$ \textit{horizontal} if at each point of $M$, the induced map between the holomorphic tangent spaces takes values in the appropriate fibre $T^{1,0}\check{D}$.
It is easy to see that the period map $\Phi: \, \mathcal{T}\rightarrow D$ is horizontal since $\Phi_*(T^{1,0}_p\mathcal{T})\subseteq T^{1,0}_{o,h}D$ for any $p\in \mathcal{T}$.
Since $D$ is an open set in $\check{D}$, we have the following relation:
\begin{align}\label{n+iso}
T^{1,0}_{o, h}D= T^{1, 0}_{o, h}\check{D}\cong\mathfrak{b}\oplus \mathfrak{g}^{-1, 1}/\mathfrak{b}\hookrightarrow \mathfrak{g}/\mathfrak{b}\cong \mathfrak{n}_+.
\end{align}
\begin{remark}\label{N+inD}With a fixed base point, we can identify $N_+$ with its unipotent orbit in $\check{D}$ by identifying an element $c\in N_+$ with $[c]=cB$ in $\check{D}$; that is, $N_+=N_+(\text{ base point })\cong N_+B/B\subseteq\check{D}.$ In particular, when the base point $o$ is in $D$, we have $N_+\cap D\subseteq D$. 
              \end{remark}
Let us introduce the notion of an adapted basis for the given Hodge decomposition or the Hodge filtration.
For any $p\in \mathcal{T}$ and $f^k=\dim F^k_p$ for any $0\leq k\leq n$, we call a basis $$\xi=\left\{ \xi_0, \xi_1, \cdots, \xi_N, \cdots, \xi_{f^{k+1}}, \cdots, \xi_{f^k-1}, \cdots,\xi_{f^{2}}, \cdots, \xi_{f^{1}-1}, \xi_{f^{0}-1} \right\}$$ of $H^n(M_p, \mathbb{C})$ an \textit{adapted basis for the given Hodge decomposition}
$$H^n(M_p, {\mathbb{C}})=H^{n, 0}_p\oplus H^{n-1, 1}_p\oplus\cdots \oplus H^{1, n-1}_p\oplus H^{0, n}_p, $$
if it satisfies $
H^{k, n-k}_p=\text{Span}_{\mathbb{C}}\left\{\xi_{f^{k+1}}, \cdots, \xi_{f^k-1}\right\}$ with $\dim H^{k,n-k}_p=f^k-f^{k+1}$.
We call a basis 
\begin{align*}
\zeta=\{\zeta_0, \zeta_1, \cdots, \zeta_N, \cdots, \zeta_{f^{k+1}}, \cdots, \zeta_{f^k-1}, \cdots, \zeta_{f^2}, \cdots, \zeta_{f^1-0}, \zeta_{f^0-1}\}
\end{align*}
of $H^n(M_p, \mathbb{C})$ an \textit{adapted basis for the given filtration} 
\begin{align*}
F^n\subseteq F^{n-1}\subseteq\cdots\subseteq F^0
\end{align*}
if it satisfies $F^{k}=\text{Span}_{\mathbb{C}}\{\zeta_0, \cdots, \zeta_{f^k-1}\}$ with $\text{dim}_{\mathbb{C}}F^{k}=f^k$. 
Moreover, unless otherwise pointed out, the matrices in this paper are $m\times m$ matrices, where $m=f^0$. The blocks of the $m\times m$ matrix $T$ is set as follows:
for each $0\leq \alpha, \beta\leq n$, the $(\alpha, \beta)$-th block $T^{\alpha, \beta}$ is
\begin{align}\label{block}
T^{\alpha, \beta}=\left[T_{ij}(\tau)\right]_{f^{-\alpha+n+1}\leq i \leq f^{-\alpha+n}-1, \ f^{-\beta+n+1}\leq j\leq f^{-\beta+n}-1},
\end{align} where $T_{ij}$ is the entries of
the matrix $T$, and $f^{n+1}$ is defined to be zero. In particular, $T =[T^{\alpha,\beta}]$ is called a \textit{block lower triangular matrix} if
$T^{\alpha,\beta}=0$ whenever $\alpha<\beta$.
\begin{remark}\label{n+}
We remark that by fixing a base point, we can identify the above quotient Lie groups or Lie algebras with their orbits in the corresponding quotient Lie algebras or Lie groups. For example, $\mathfrak{n}_+\cong \mathfrak{g}/\mathfrak{b}$, $\mathfrak{g}^{-1,1}\cong\mathfrak{b}\oplus \mathfrak{g}^{-1,1}/\mathfrak{b}$, and $N_+\cong N_+B/B\subseteq \check{D}$.
We can also identify a point $\Phi(p)=\{ F^n_p\subseteq F^{n-1}_p\subseteq \cdots \subseteq F^{0}_p\}\in D$ with its Hodge decomposition $\bigoplus_{k=0}^n H^{k, n-k}_p$, and thus with any fixed adapted basis of the corresponding Hodge decomposition for the base point, we have matrix representations of elements in the above Lie groups and Lie algebras. For example, elements in $N_+$ can be realized as nonsingular block lower triangular matrices with identity blocks in the diagonal; elements in $B$ can be realized as nonsingular block upper triangular matrices. 
\end{remark}

 We shall review and collect some facts about the structure of simple Lie algebra $\mathfrak{g}$ in our case. Again one may refer to \cite{GS} and \cite{schmid1} for more details. Let $\theta: \,\mathfrak{g}\rightarrow \mathfrak{g}$ be the Weil operator, which is defined by 
\begin{align*}\theta(X)=(-1)^p X\quad \text{ for } X\in \mathfrak{g}^{p,-p}.
\end{align*}
Then $\theta$ is an involutive automorphism of $\mathfrak{g}$, and is defined over $\mathbb{R}$. The $(+1)$ and $(-1)$ eigenspaces of $\theta$ will be denoted by $\mathfrak{k}$ and $\mathfrak{p}$ respectively. Moreover, set 
\begin{align*}
\mathfrak{k}_0=\mathfrak{k}\cap \mathfrak{g}_0, \quad \mathfrak{p}_0=\mathfrak{p}\cap \mathfrak{g}_0. 
\end{align*}The fact that $\theta$ is an involutive automorphism implies 
\begin{align*}
\mathfrak{g}=\mathfrak{k}\oplus\mathfrak{p}, \quad \mathfrak{g}_0=\mathfrak{k}_0\oplus \mathfrak{p}_0, \quad
[\mathfrak{k}, \,\mathfrak{k}]\subseteq \mathfrak{k}, \quad [\mathfrak{p},\,\mathfrak{p}]\subseteq\mathfrak{p}, \quad [\mathfrak{k}, \,\mathfrak{p}]\subseteq \mathfrak{p}.
\end{align*}
Let us consider  
$\mathfrak{g}_c=\mathfrak{k}_0\oplus \sqrt{-1}\mathfrak{p}_0.$
Then $\mathfrak{g}_c$ is a real form for $\mathfrak{g}$. 
Recall that the killing form $B(\cdot, \,\cdot)$ on $\mathfrak{g}$ is defined by
\begin{align*}B(X,Y)=\text{Trace}(\text{ad}(X)\circ\text{ad}(Y)) \quad \text{for } X,Y\in \mathfrak{g}. 
\end{align*}
A semisimple Lie algebra is compact if and only if the Killing form is negative definite. Thus it is not hard to check that $\mathfrak{g}_c$ is actually a compact real form of $\mathfrak{g}$, while $\mathfrak{g}_0$ is a non-compact real form. 
Recall that $G_{\mathbb{R}}\subseteq G_{\mathbb{C}}$ is the subgroup which correpsonds to the subalgebra $\mathfrak{g}_0\subseteq\mathfrak{g}$. Let us denote the connected subgroup $G_c\subseteq G_{\mathbb{C}}$ which corresponds to the subalgebra $\mathfrak{g}_c\subseteq\mathfrak{g}$. Let us denote the complex conjugation of $\mathfrak{g}$ with respect to the compact real form by $\tau_c$, and the complex conjugation of $\mathfrak{g}$ with respect to the compact real form by $\tau_0$. 

The intersection $K=G_c\cap G_{\mathbb{R}}$ is then a compact subgroup of $\mathbb{G}_{\mathbb{R}}$, whose Lie algebra is $\mathfrak{k}_0=\mathfrak{g}_{\mathbb{R}}\cap \mathfrak{g}_c$. 
With the above notations, Schmid showed in \cite{schmid1} that $K$ is a maximal compact subgroup of $G_{\mathbb{R}}$, and it meets every connected component of $G_{\mathbb{R}}$. Moreover, $V=G_{\mathbb{R}}\cap B\subseteq K$. 
As remarked in $\S1$ in \cite{GS} of Griffiths and Schmid, one gets that $\mathfrak{v}$ must have the same rank of $\mathfrak{g}$ as $\mathfrak{v}$ is  the intersection of the two parabolic subalgebras $\mathfrak{b}$ and $\tau_c(\mathfrak{b})$. Moreover, $\mathfrak{g}_0$ and $\mathfrak{v}_0$ are also of equal rank, since they are real forms of $\mathfrak{g}$ and $\mathfrak{v}$ respectively. Therefore, we have the following proposition. 
\begin{proposition}\label{cartaninK}There exists a Cartan subalgebra $\mathfrak{h}_0$ of $\mathfrak{g}_0$ such that $\mathfrak{h}_0\subseteq \mathfrak{v}_0\subseteq \mathfrak{k}_0$ and $\mathfrak{h}_0$ is also a Cartan subalgebra of $\mathfrak{k}_0$. \end{proposition} 
Proposition \ref{cartaninK} implies that the simple Lie algebra $\mathfrak{g}_0$ in our case is a simple Lie algebra of first category as defined in $\S4$ in \cite{Sugi}. In the upcoming part, we will briefly derive the result of a simple Lie algebra of first category in Lemma 3 in \cite{Sugi1}. One may also refer to \cite{Xu} Lemma 2.2.12 at pp.~141-142 for the same result. 

Let us still use the above notations of the Lie algebras we consider. By Proposition 4, we can take $\mathfrak{h}_0$ to be a Cartan subalgebra of $\mathfrak{g}$ such that $\mathfrak{h}_0\subseteq \mathfrak{v}_0\subseteq\mathfrak{k}_0$ and $\mathfrak{h}_0$ is also a Cartan subalgebra of $\mathfrak{k}_0$. Let us denote $\mathfrak{h}$ to be the complexification of $\mathfrak{h}_0$. Then $\mathfrak{h}$ is a Cartan subalgebra of $\mathfrak{g}$ such that $\mathfrak{h}\subseteq \mathfrak{v}\subseteq \mathfrak{k}$. 

Write $\mathfrak{h}_0^*=\text{Hom}(\mathfrak{h}_0, \mathbb{R})$ and $\mathfrak{h}^*_{_\mathbb{R}}=\sqrt{-1}\mathfrak{h}^*_0. $ Then $\mathfrak{h}^*_{_\mathbb{R}}$ can be identified with $\mathfrak{h}_{_\mathbb{R}}:=\sqrt{-1}\mathfrak{h}_0$ by duality using the restriction of the Killing form $B$ of $\mathfrak{g}$ to $\mathfrak{h}_{_\mathbb{R}}$. 
 Let $\rho\in\mathfrak{h}_{_{\mathbb{R}}}^*\simeq \mathfrak{h}_{_{\mathbb{R}}}$, one can define the following subspace of $\mathfrak{g}$
\begin{align*}
\mathfrak{g}^{\rho}=\{x\in \mathfrak{g}| [h, \,x]=\rho(h)x\quad\text{for all } h\in \mathfrak{h}\}.
\end{align*}
An element $\varphi\in\mathfrak{h}_{_{\mathbb{R}}}^*\simeq\mathfrak{h}_{_{\mathbb{R}}}$ is called a root of $\mathfrak{g}$ with respect to $\mathfrak{h}$ if $\mathfrak{g}^\varphi\neq \{0\}$. 

Let $\Delta\subseteq \mathfrak{h}^{*}_{_\mathbb{R}}\simeq\mathfrak{h}_{_{\mathbb{R}}}$ denote the space of nonzero $\mathfrak{h}$-roots. Then each root space
\begin{align*}
\mathfrak{g}^{\varphi}=\{x\in\mathfrak{g}|[h,x]=\varphi(h)x \text{ for all } h\in \mathfrak{h}\}
\end{align*} belongs to some $\varphi\in \Delta$ is one-dimensional over $\mathbb{C}$, generated by a root vector $e_{_{\varphi}}$. 

Since the involution $\theta$ is a Lie-algebra automorphism fixing $\mathfrak{k}$, we have $[h, \theta(e_{_{\varphi}})]=\varphi(h)\theta(e_{_{\varphi}})$ for any $h\in \mathfrak{h}$ and $\varphi\in \Delta.$ Thus $\theta(e_{_{\varphi}})$ is also a root vector belonging to the root $\varphi$, so $e_{_{\varphi}}$ must be an eigenvector of $\theta$. It follows that there is a decomposition of the roots $\Delta$ into $\Delta_{\mathfrak{k}}\cup\Delta_{\mathfrak{p}}$ of compact roots and non-compact roots with root spaces $\mathbb{C}e_{_{\varphi}}\subseteq \mathfrak{k}$ and $\mathfrak{p}$ respectively. 
The adjoint representation of $\mathfrak{h}$ on $\mathfrak{g}$ determins a decomposition 
\begin{align*}
\mathfrak{g}=\mathfrak{h}\oplus \sum_{\varphi\in\Delta}\mathfrak{g}^{\varphi}.
\end{align*}There also exists a Weyl base $\{h_i, 1\leq i\leq l;\,\,e_{_{\varphi}}, \text{ for any } \varphi\in\Delta\}$ with $l=\text{rank}(\mathfrak{g})$ such that
$\text{Span}_{\mathbb{C}}\{h_1, \cdots, h_l\}=\mathfrak{h}$, $\text{Span}_{\mathbb{C}}\{e_{\varphi}\}=g^{\varphi}$ for each $\varphi\in \Delta$, and
\begin{align*}
&\tau_c(h_i)=\tau_0(h_i)=-h_i, \quad \text{ for any }1\leq i\leq l;\\ \tau_c(e_{_{\varphi}})=&\tau_0(e_{_{\varphi}})=-e_{_{-\varphi}} \quad \text{for any } \varphi\in \Delta_{\mathfrak{k}};\quad\tau_0(e_{_{\varphi}})=-\tau_c(e_{_{\varphi}})=e_{_{\varphi}} \quad\text{for any }\varphi\in \Delta_{\mathfrak{p}}.
\end{align*}With respect to this Weyl base, we have 
\begin{align*}
\mathfrak{k}_0&=\mathfrak{h}_0+\sum_{\varphi\in \Delta_{\mathfrak{k}}}\mathbb{R}(e_{_{\varphi}}-e_{_{-\varphi}})+\sum_{\varphi\in\Delta_{\mathfrak{k}}} \mathbb{R}\sqrt{-1}(e_{_{\varphi}}+e_{_{-\varphi}});\\
\mathfrak{p}_0&=\sum_{\varphi\in \Delta_{\mathfrak{p}}}\mathbb{R}(e_{_{\varphi}}+e_{_{-\varphi}})+\sum_{\varphi\in\Delta_{\mathfrak{p}}} \mathbb{R}\sqrt{-1}(e_{_{\varphi}}-e_{_{-\varphi}}).
\end{align*}

Let us now introduce a lexicographic order (cf. pp.41 in \cite{Xu} or pp.416 in \cite{Sugi}) in the real vector space $\mathfrak{h}_{_{\mathbb{R}}}$ as follows: we fix an ordered basis $e_1, \cdots, e_l$ for $\mathfrak{h}_{_{\mathbb{R}}}$. Then for any $h=\sum_{i=1}^l\lambda_ie_i\in\mathfrak{h}_{_{\mathbb{R}}}$, we call $h>0$ if the first nonzero coefficient is positive, that is, if $\lambda_1=\cdots=\lambda_k=0, \lambda_{k+1}>0$ for some $1\leq k<l$. For any $h, h'\in\mathfrak{h}_{_{\mathbb{R}}}$, we say $h>h'$ if $h-h'>0$, $h<h'$ if $h-h'<0$ and $h=h'$ if $h-h'=0$.
Now let us first choose a \textit{maximal} linearly independent subset $S=\{s_1, \cdots, s_k\}$ of $\Delta_{\mathfrak{k}}$, and then choose a linearly independent subset $E=\{e_1, \cdots, e_{l-k}\}$ of $\Delta_{\mathfrak{p}}$ such that $E\cup S$ is a basis for $\mathfrak{h}^*_{\mathbb{R}}$, where $l$ is the real dimension of $\mathfrak{h}^*_{\mathbb{R}}$. Now we order this basis $E\cup S$ as $\{e_1, \cdots, e_{l-k}, s_1, \cdots, s_k\}$, namely, we put the noncompact roots in front of the compact ones. Then we define the above lexicographic order in $\mathfrak{h}^*_{\mathbb{R}}\simeq \mathfrak{h}_{\mathbb{R}}$. 
Then we define $\Delta^{\pm}$, $\Delta^{\pm}_{\mathfrak{p}}$, and $\Delta^{\pm}_{\mathfrak{k}}$.
\begin{definition}Two different roots $\varphi, \psi\in \Delta$ are said to be strongly orthogonal if and only if $\varphi\pm\psi\notin\Delta\cup \{0\}$, which is denoted by $\varphi \independent\psi$. 
\end{definition}
For the real simple Lie algebra $\mathfrak{g}_0=\mathfrak{k}_0\oplus\mathfrak{p}_0$ which has a Cartan subalgebra $\mathfrak{h}_0$ in $\mathfrak{k}_0$, the maximal abelian subspace of $\mathfrak{p}_0$ can be described as in the following lemma, which is a slight extension of a lemma of Harish-Chandra in \cite{HC}. One may refer to Lemma 3 in \cite{Sugi1} or Lemma 2.2.12 at pp.141--142 in \cite{Xu} for more details. For reader's convenience we give the detailed proof.
\begin{lemma}\label{stronglyortho}There exists a set of strongly orthogonal noncompact positive roots $\Lambda=\{\varphi_1, \cdots, \varphi_r\}\subseteq\Delta^+_{\mathfrak{p}}$ such that 
\begin{align*}
\mathfrak{a}_0=\sum_{i=1}^r\mathbb{R}\left(e_{_{\varphi_i}}+e_{_{-\varphi_i}}\right)
\end{align*} 
is a maximal abelian subspace in $\mathfrak{p}_0$. 
\end{lemma}
This lemma is a slight extension of a lemma of Harish-Chandra in \cite{HC}. One may refer to Lemma 3 in \cite{Sugi1} or Lemma 2.2.12 at pp.141--142 in \cite{Xu} for more details. One may also find the detailed proof in \cite{LS}.


\end{document}